\documentclass[a4paper,10pt]{article}
\usepackage{amsfonts,amsmath,amssymb,amsthm,verbatim,placeins,comment,placeins,caption}
\usepackage[dvips]{graphicx}
\usepackage{natbib}
\usepackage{bbm}
\usepackage{color}
\usepackage[usenames,dvipsnames,svgnames,table]{xcolor}
\usepackage{geometry}
\usepackage[cp1250]{inputenc}
\usepackage{lmodern}
\usepackage[OT2,T1]{fontenc}
\usepackage{graphicx}
\usepackage{caption,subfig}
\usepackage{fancyhdr}
\usepackage[style=english]{csquotes}

\newtheorem{theorem}{Theorem}[section]
\newtheorem{lemma}[theorem]{Lemma}

\newtheorem{cor}[theorem]{Corollary}

\theoremstyle{definition}

\theoremstyle{remark}
\newtheorem{remark}[theorem]{Remark}

\numberwithin{equation}{section}

\newcommand{\N}{\mathbb{N}}
\newcommand{\R}{\mathbb{R}}

\newcommand{\n}{\mathop{\mathrm{\mathfrak{n}}}}

\date{}

\begin{document}

\title{\textbf{Fractional Bessel Process with Constant Drift: Spectral Analysis and Queueing Applications}}
\maketitle

\centerline{\author{\textbf{I. Papi\'c}$^{\textrm{*}}$}}

{\footnotesize{
$$\begin{tabular}{l}
 
  $^{\textrm{*}}$ \emph{School of Applied Mathematics and Informatics, J.J. Strossmayer University of Osijek, Osijek, Croatia} \\
\end{tabular}$$}}
\bigskip

\noindent\textbf{Abstract.}
We introduce a fractional Bessel process with constant negative drift, defined as a time-changed Bessel process via the inverse of a stable subordinator, independent of the base process. This construction yields a model capable of capturing subdiffusive behavior and long-range dependence, relevant in various complex systems. We derive an explicit spectral representation of its transition density, extending the non-fractional setting. Using this representation, we establish several analytical properties of the process, including its stationary distribution and correlation structure. These results provide new insights into the behavior of fractional diffusions and offer analytical tools for applications in queueing theory, mathematical finance, and related domains. In particular, we demonstrate their applicability through a concrete problem in queueing theory.

\bigskip

\noindent\textbf{Key words.} Bessel process, Fractional Bessel process, Mittag-Leffler function, Spectral representation, Transition density, Whittaker function, Queueing theory, Heavy-traffic approximation.

\bigskip

\noindent\textbf{Mathematics Subject Classification (2020):} 60G22, 60J60, 35R11, 35P10, 60K25.

\bigskip

\section{Introduction} \label{intro}
The Bessel process with constant drift was first studied in detail by Linetsky (\cite{linetsky_2004}), who explicitly calculated the spectral representation of its transition density, making it statistically tractable. His motivation for studying this particular stochastic process stemmed from its applications in mathematical finance and queueing theory. In particular, \cite{Coffman1998} presents the well-known heavy-traffic averaging principle (HTAP). In queueing theory language, this principle states that the work in each queue is emptied and refilled at a rate much faster than the rate at which the total workload changes. This implies that the total workload can be treated as approximately constant during the course of a cycle, while the individual queue workloads fluctuate in a fluid-like manner.

In polling systems (multi-queue systems where a single server visits queues in a fixed order), HTAP leads to a diffusion approximation involving the Bessel process with constant drift. Additionally, Linetsky shows that this process can be transformed into well-known models in mathematical finance, such as the $3/2$ model and the reciprocal of the square-root CIR process, thereby facilitating tractable pricing methods, for example, in the evaluation of Asian options (\cite{linetsky_2004}).

Motivated by the analytical tractability and wide applicability of this process, we propose a new time-changed version of the Bessel process with constant drift. In our formulation, the time change is governed by the inverse of a stable subordinator, independent of the outer process. This extension introduces flexibility into the model and enables it to capture subdiffusive behavior and long-range dependence, phenomena frequently observed in complex systems. Such time-changed processes arise naturally when modeling power-law waiting times between particle jumps, and have proven useful across disciplines such as hydrology (\cite{chakraborty2009parameter}), biology (\cite{leonenko2025time}), and finance (\cite{fontana2023cbi}). For a broader perspective on the interplay between stochastic modeling, fractional calculus, and differential equations, see \cite{meerschaert2019stochastic}.

Since Linetsky's foundational work, the Bessel process with constant drift has continued to attract significant interest, both theoretical and applied. Alternative closed-form expressions for its transition density have been derived in recent works (\cite{Guarnieri2017}, \cite{Zhenyu2016}), though our analysis builds on Linetsky’s spectral approach. Moreover, the process has appeared as a special case of more general models driven by fractional Brownian motion (\cite{Bishwal2021}). Its applications span diverse domains: in climatology, it has been used to describe Arctic sea-ice thickness dynamics (\cite{toppaladoddi2015theory, toppaladoddi2017statistical}); in biology, it models the dynamics of DNA bubble formation under thermal fluctuations (\cite{fogedby2007dynamics}); and in bird navigation, it appears under the name “Pole-seeking Brownian motion” (\cite{kendall1974pole}). In conformal quantum mechanics, the process is linked to quantum Hamiltonians with Coulomb potentials (\cite{rajabpour2009bessel}), which is also noted by Linetsky. Further applications in mathematical finance and economics can be found in \cite{Burtnyak2017} and \cite{li2016bessel}.

The structure of the paper is as follows. In Section \ref{Bessel_process}, we briefly review key properties of the Bessel process with constant drift. Section \ref{frac_Bessel_process} introduces the fractional Bessel process (with constant drift) and presents our main results. Specifically, we derive the spectral representation of its transition density and provide a strong solution to the corresponding fractional Cauchy problem. We also demonstrate that the stationary distribution remains the same as in the non-fractional case (a gamma distribution). Furthermore, the spectral representation enables us to compute the correlation structure of the (fractional) Bessel process, an expression not previously available in the literature, even for the non-fractional case. As a result, we establish that the process exhibits long-range dependence. Finally, in Section \ref{Applications}, we present concrete application of the model in queueing theory.
\section{Bessel process}\label{Bessel_process}
The Bessel process with constant negative drift was studied in detail in \cite{linetsky_2004}, where the spectral representation of the associated transition density was explicitly computed. The defining stochastic differential equation (SDE) for this process is
\begin{equation}\label{Bessel_SDE}
	dX(t)=\left(\frac{\nu+\frac{1}{2}}{X(t)}+\mu\right)dt + dB(t), \quad \nu \in \langle -1, +\infty \rangle \backslash \{-1/2\}, \quad \mu \in \R \backslash \{0\},
\end{equation}
where $(B(t), \, t \geq 0)$ denotes a standard Brownian motion. The state space of this process is   $\left[ 0, +\infty\right>$, and its parameters are $\mu$ and $\nu$. The infinitesimal generator of the process is given by
\begin{equation}\label{GeneratorBessel}
	\mathcal{G}f(x)=\left(\frac{\nu+\frac{1}{2}}{x}+\mu\right)f'(x)+\frac{1}{2}f''(x)
\end{equation}
with the domain
     \begin{equation}\label{BesselDomain}
	D(\mathcal{G})=\bigg \{ f \in L^{2}\left(\left[ 0, +\infty\right>, \mathfrak{m}\right) \cap C^2\left( \left[ 0, +\infty\right> \right) \colon \mathcal{G}f \in L^{2}\left(\left[ 0, +\infty\right>, \mathfrak{m} \right), \quad  \lim_{x \downarrow 0}\frac{f'(x)}{\mathfrak{s}(x)}=0\bigg\},
\end{equation}
where the speed and scale densities are
\begin{equation*}
	\mathfrak{s}(x)=x^{-2\nu-1}e^{-2\mu x}, \quad \mathfrak{m}(x)=2x^{2\nu +1}e^{2\mu x}.
\end{equation*}
If the parameter $\mu$ is positive, the process is transient, while for negative $\mu$, the process admits a stationary distribution with gamma density:
\begin{equation}\label{stat_dist_Bessel}
	\pi(x)=\frac{-2\mu}{\Gamma(2\nu+2)}\left(-2\mu x\right)^{2\nu+1}e^{2\mu x}, \quad x \geq 0.
\end{equation}
We fully adapt the notation in \cite{linetsky_2004}, as our main result is an extension of that work.

The spectral representation of a stochastic process in this context involves solving the eigenvalue problem for the generator $\mathcal{G}$ (or equivalently, for the Sturm–Liouville operator $(-\mathcal{G})$):
\begin{equation*}
	(-\mathcal{G})f=\lambda f, \quad \lambda >0.
\end{equation*}
In general, the spectrum consists of two disjoint components: a discrete and a continuous (essential) part (see \cite{simon1980methods} and \cite{linetsky2004spectral} for details).

In our setting, the discrete spectrum of $\mathcal{G}$, given by \eqref{GeneratorBessel}, depends on the parameter values:
\begin{itemize}
	\item For $(\nu, \mu) \in \langle -1/2, +\infty \rangle \times \langle -\infty, 0 \rangle$:
	\begin{equation*}
	\sigma_d(\mathcal{G})=\left\{\frac{\mu^2}{2}\left(1-\left(\frac{\nu+\frac{1}{2}}{n+\nu+\frac{1}{2}}\right)^2\right), \,n \in \N_0\right\},
\end{equation*}	 
with eigenvalues
	\begin{equation}\label{Bessel_eigenvalue}
		\lambda_n=\frac{\mu^2}{2}\left(1-\left(\frac{\nu+\frac{1}{2}}{n+\nu+\frac{1}{2}}\right)^2\right), \quad n \in \mathbb{N}_0,
	\end{equation}
and eigenfunctions related to generalized Laguerre polynomials $L_n^{2\nu}(x)$, which solve:
	\begin{equation*}
		xy''+\left(2\nu+1-x\right)y'+ny=0, \, n \in \N_0.
	\end{equation*}
The corresponding eigenfunctions are
\begin{equation} \label{eigenfunction_discrete}
\Psi_n(x)=\frac{n! \Gamma(1+2\nu)}{\Gamma(1+2\nu+n)}e^{-\frac{n\mu x}{\left(\nu+n+1/2\right)}}\left(-2 \frac{\mu \left(\nu+1/2\right)}{\nu+n+1/2}\right)^{\nu+1/2} L^{2\nu}_n \left(-2x \frac{\mu \left(\nu+1/2\right)}{\nu+n+1/2}\right).
\end{equation}
Note that the eigenvalues cluster near the constant $\mu^2/2$.
	\item For $(\nu, \mu) \in \langle -1, -1/2 \rangle \times \langle -\infty, 0 \rangle$: 
		\begin{equation*}
		\sigma_d(\mathcal{G})=\left\{0\right\}.
	\end{equation*}	 
	\item For $(\nu, \mu) \in \langle -1/2, +\infty \rangle \times \langle 0, \infty \rangle$: 
		\begin{equation*}
	\sigma_d(\mathcal{G})=\emptyset.
	\end{equation*}	 
	\item For $(\nu, \mu) \in \langle -1, -1/2 \rangle \times \langle  0, \infty \rangle$:
		\begin{equation*}
		\sigma_d(\mathcal{G})=\left\{\frac{\mu^2}{2}\left(1-\left(\frac{\nu+\frac{1}{2}}{n+\nu+\frac{1}{2}}\right)^2\right), \,n \in \N_0\right\}.
	\end{equation*}	
\end{itemize}

The essential part of the spectrum corresponds to the continuous spectrum,
\begin{equation*}\sigma_c(\mathcal{G})=\left[\frac{\mu^2}{2}, +\infty\right>,
\end{equation*}
while functions associated with essential part of the spectrum are
\begin{align} 
	\Psi(x, \lambda)=x^{-\nu-1/2} e^{-\mu x}M_{i\beta/\sqrt{2\lambda-\mu^2}, \nu}(-2i\sqrt{2\lambda-\mu^2} x), \label {eigenfunction_essential}\\
	 \hat{\Psi}(x, \lambda)=x^{-\nu-1/2} e^{-\mu x} W_{i\beta/\sqrt{2\lambda-\mu^2}, \nu}(-2i\sqrt{2\lambda-\mu^2} x),
\end{align}
where 
\begin{align*}
	M_{\chi, \nu}(z):&=z^{\nu+1/2}e^{-z/2} M(\nu-\chi+1/2;1+2\nu,z),  \\
		W_{\chi, \nu}(z):&=z^{\nu+1/2}e^{-z/2} U(\nu-\chi+1/2;1+2\nu,z), \quad -\pi<arg(z)\leq \pi,
\end{align*}
are the Whittaker functions with $M(a,b;z)$ and $U(a,b;z)$ being first and second confluent hypergeometric functions (for details about Whittaker and related special functions see e.g. \cite{Buchholz}).

The Bessel process is a time-homogeneous diffusion. Let $p(x, t; y)$ denote its transition density:
\begin{equation*}
	p(x, t; y) = \frac{d}{dx}P(X(t) \leq x |\, X(0)=y).
\end{equation*}
Define
\begin{equation}\label{beta_def}
	\beta:=-\mu\left(\nu+\frac{1}{2}\right).
\end{equation}
According to \cite{linetsky_2004}, the spectral representation of transition density $p(x, t; y)$ of Bessel process  with negative constant drift (as defined by \eqref{Bessel_SDE}) is given by
\begin{equation}\label{Bessel_trans_dens}
	p(x,t;y)=p_d(x,t;y)+p_c(x,t;y),
\end{equation}
where $p_d$ corresponds to the discrete spectrum contribution and $p_c$ to the continuous one. The explicit expressions for \( p_d \) and \( p_c \) vary depending on the values of \( \nu \) and \( \mu \), and are given below. They reflect the decomposition of the transition density into discrete and continuous spectral components associated with the generator \( \mathcal{G} \).

\[
p_d(x,t;y)=
\begin{cases}
	\pi(x)+\sum\limits_{n=1}^{\infty}\frac{n!\beta\left(2\beta x\right)^{2\nu+1}}{\left(n+\nu+\frac{1}{2}\right)^{2\nu+3}\Gamma(2\nu+n+1)} \\ \times e^{2\mu x-\frac{\mu n}{n+\nu+\frac{1}{2}}(x+y)-\lambda_n t}L^{2\nu}_{n}\left(\frac{2 \beta x}{n+\nu+\frac{1}{2}}\right)L^{2\nu}_{n}\left(\frac{2 \beta y}{n+\nu+\frac{1}{2}}\right) &,\, (\nu, \mu) \in \langle -1/2, +\infty \rangle \times \langle-\infty, 0 \rangle\\
	\pi(x) &,\, (\nu, \mu) \in \langle -1, -1/2 \rangle \times \langle-\infty, 0 \rangle\\
	0 &,\, (\nu, \mu) \in \langle -1/2, +\infty \rangle \times \langle 0, +\infty \rangle\\
	\sum\limits_{n=1}^{\infty}\frac{n!\beta\left(2\beta x\right)^{2\nu+1}}{\left(n+\nu+\frac{1}{2}\right)^{2\nu+3}\Gamma(2\nu+n+1)}\\ \times e^{2\mu x-\frac{\mu n}{n+\nu+\frac{1}{2}}(x+y)-\lambda_n t} L^{2\nu}_{n}\left(\frac{2 \beta x}{n+\nu+\frac{1}{2}}\right)L^{2\nu}_{n}\left(\frac{2 \beta y}{n+\nu+\frac{1}{2}}\right) &,\, (\nu, \mu) \in \langle -1, -1/2 \rangle \times \langle 0, +\infty \rangle.
\end{cases}
\]
and
	\begin{align*}
	p_{c}(x,t;y)=&\frac{1}{2\pi}\int\limits_{\mu^2/2}^{\infty}\frac{1}{k(\lambda)}\left|\frac{\Gamma(\frac{1}{2}+\nu+i\beta/ k(\lambda))}{\Gamma(1+2\nu)}\right |^2 e^{-\lambda t}\left(\frac{x}{y}\right)^{\nu+\frac{1}{2}}e^{\mu(x-y)+\pi \beta /k(\lambda)} \\
	&\times M_{i\beta/ k(\lambda),\, \nu}(-2ik(\lambda) y) M_{-i\beta/ k(\lambda), \, \nu}(2ik(\lambda) x)\, d\lambda,
\end{align*}
for $y>0$, while
\[
p_d(x,t;0)=
\begin{cases}
	\pi(x)+\sum\limits_{n=1}^{\infty}\frac{\beta\left(2\beta x\right)^{2\nu+1}}{\left(n+\nu+\frac{1}{2}\right)^{2\nu+3}\Gamma(2\nu+1)} \\ \times e^{2\mu x-\frac{\mu n}{n+\nu+\frac{1}{2}}x-\lambda_n t}L^{2\nu}_{n}\left(\frac{2 \beta x}{n+\nu+\frac{1}{2}}\right) &,\, (\nu, \mu) \in \langle -1/2, +\infty \rangle \times \langle-\infty, 0 \rangle\\
	\pi(x) &,\, (\nu, \mu) \in \langle -1, -1/2 \rangle \times \langle-\infty, 0 \rangle\\
	0 &,\, (\nu, \mu) \in \langle -1/2, +\infty \rangle \times \langle 0, +\infty \rangle\\
	\sum\limits_{n=1}^{\infty}\frac{\beta\left(2\beta x\right)^{2\nu+1}}{\left(n+\nu+\frac{1}{2}\right)^{2\nu+3}\Gamma(2\nu+1)}\\ \times e^{2\mu x-\frac{\mu n}{n+\nu+\frac{1}{2}}x-\lambda_n t} L^{2\nu}_{n}\left(\frac{2 \beta x}{n+\nu+\frac{1}{2}}\right) &,\, (\nu, \mu) \in \langle -1, -1/2 \rangle \times \langle 0, +\infty \rangle.
\end{cases}
\]
and
\begin{align*}
	p_{c}(x,t;0)=&\frac{1}{2\pi}\int\limits_{\mu^2/2}^{\infty}2^{\nu+1/2}k(\lambda)^{\nu - 1/2}e^{i\pi (\nu/2+1/4)}\left|\frac{\Gamma(\frac{1}{2}+\nu+i\beta/ k(\lambda))}{\Gamma(1+2\nu)}\right |^2 e^{-\lambda t}x^{\nu+\frac{1}{2}}e^{\mu x+\pi \beta /k(\lambda)} \\
	&\times M_{i\beta/ k(\lambda), \, \nu}(-2ik(\lambda) x)\, d\lambda,
\end{align*} for $y=0$,
where $k(\lambda)=\sqrt{2\lambda-\mu^2}$
(for details and the proof see Proposition 1., \cite{linetsky_2004}).

\section{Fractional Bessel process}\label{frac_Bessel_process}
Let \( (D(t),\, t \geq 0) \) be a standard stable subordinator with index \( 0 < \alpha < 1 \). This is a non-decreasing Lévy process characterized by the Laplace transform
\begin{equation*}
\mathbb{E}[e^{-s D(t)}] = \exp(-t s^{\alpha}), \quad s \geq 0.
\end{equation*}
We define the inverse subordinator (also known as the first passage or hitting time process) by
\begin{equation}\label{InverseStabDef}
	E(t) = \inf \{ x > 0 : D(x) > t \}, \quad t \geq 0.
\end{equation}
The process $(E(t), \,t \geq 0)$ is non-Markovian and non-decreasing. For every fixed \( t \geq 0 \), the random variable \( E(t) \) has a probability density function \( f_t(x) \) whose Laplace transform is given by
\begin{equation} \label{Mittag}
	\mathbb{E}[e^{-s E(t)}] = \int_{0}^{\infty} e^{-s x} f_t(x)\, dx = \mathcal{E}_{\alpha}(-s t^{\alpha}),
\end{equation}
where
\begin{equation} \label{Mittag_Leffler_function_definition}
	\mathcal{E}_{\alpha}(z) := \sum_{j=0}^{\infty} \frac{z^j}{\Gamma(1 + \alpha j)}, \quad z \in \mathbb{C},
\end{equation}
is the Mittag-Leffler function, which converges absolutely for all complex $z$ (see, e.g., \cite{gorenflo2016mittag}).


The fractional Bessel process is defined as the time-changed Bessel process via the inverse of the stable subordinator:
\begin{equation}\label{fracBesselprocess}
	X_{\alpha}(t):=X(E(t)),
\end{equation}
where $(X(t), \, t \geq 0)$ is the Bessel process corresponding to the SDE \eqref{Bessel_SDE} and $(E(t), \, t \geq 0)$ is the inverse $\alpha$-stable subordinator \eqref{InverseStabDef}, assumed to be independent of the Bessel process. We emphasize that the resulting process is non-Markovian, but nevertheless the function $p_{\alpha}(x;t,y)$ such that
\begin{equation*}
	P(X_{\alpha}(t) \in B | X_{\alpha}(0)=y)=\int_{B}p_{\alpha}(x,t;y)dx, \text{ for any Borel subset B of } \left[0,+\infty\right>
\end{equation*}
will be referred to as transition density of the fractional Bessel process.
\subsection{Spectral representation of transition density of the fractional Bessel process}
In order to prove the main result, we first state and prove two technical lemmas.
\begin{lemma}\label{Lemma1FBP}
	Let $\nu \in \langle -1, +\infty \rangle \backslash \{-\frac{1}{2}\}, \, \mu \in \R$ and $\beta$ given by \eqref{beta_def}, with $\mathcal{E}_{\alpha}(\cdot)$ denoting Mittag-Leffler function and $L^{2\nu}_{n}(\cdot)$ the generalized Laguerre polynomial. Then for $x \geq 0, y >0$
\begin{equation*}
		\sum\limits_{n=1}^{\infty}\frac{n!\beta\left(2\beta x\right)^{2\nu+1}}{\left(n+\nu+\frac{1}{2}\right)^{2\nu+3}\Gamma(2\nu+n+1)}e^{2\mu x-\frac{\mu n}{n+\nu+\frac{1}{2}}(x+y)}\mathcal{E}_{\alpha}(-\lambda_n t^{\alpha})\left|L^{2\nu}_{n}\left(\frac{2 \beta x}{n+\nu+\frac{1}{2}}\right) L^{2\nu}_{n}\left(\frac{2 \beta y}{n+\nu+\frac{1}{2}}\right)\right| < \infty,
\end{equation*}
while for $x \geq 0$
\begin{equation*}
	\sum\limits_{n=1}^{\infty}\frac{\beta\left(2\beta x\right)^{2\nu+1}}{\left(n+\nu+\frac{1}{2}\right)^{2\nu+3}\Gamma(2\nu+1)}e^{2\mu x-\frac{\mu n}{n+\nu+\frac{1}{2}}x}\mathcal{E}_{\alpha}(-\lambda_n t^{\alpha})\left|L^{2\nu}_{n}\left(\frac{2 \beta x}{n+\nu+\frac{1}{2}}\right) \right| < \infty.
\end{equation*}
\end{lemma}
\begin{proof}
	In order to prove the convergence of the series, it is sufficient to show that for large values of $n$, the terms are bounded by those of a convergent series.
	
	According to \cite{OrthogonalFunctions} (p. 348 expression (36)), we have
	\begin{equation*}
		L^{2\nu}_{n}\left(\frac{2 \beta x}{n+\nu+\frac{1}{2}}\right) \thicksim O\left(\left(2\beta x\right)^{-1/4-\nu}n^{2\nu}\right), \quad n \to \infty.
	\end{equation*}
On the other hand, the sequence $\left(\lambda_n, \, n \in \N \right)$, where $\lambda_n$ is given by \eqref{Bessel_eigenvalue}, is an increasing sequence of eigenvalues such that
\begin{equation*}
	\lim_{n \to \infty}\lambda_n=\frac{\mu^2}{2}.
\end{equation*}
 Therefore, the Mittag-Leffler function $\mathcal{E}_{\alpha}(-\lambda_n t^{\alpha})$, which is a decreasing and bounded function (with respect to $n$), is also bounded for large values of $n$.
 
 Now using the asymptotic relation
 \begin{equation*}
 	\Gamma(2\nu+n+1) \thicksim \Gamma(n)\cdot n^{2\nu+1}, \quad n \to \infty
 \end{equation*}
we obtain
\begin{equation*}
	\frac{n!\beta\left(2\beta x\right)^{2\nu+1}}{\left(n+\nu+\frac{1}{2}\right)^{2\nu+3}\Gamma(2\nu+n+1)}e^{2\mu x-\frac{\mu n}{n+\nu+\frac{1}{2}}(x+y)} \thicksim O\left(\beta \frac{ \left(2\beta x\right)^{2\nu+1}}{n^{4\nu+3}}e^{\mu (x-y)}\right), \quad n \to \infty.
\end{equation*}
Taking all of the above estimates into account, we conclude

\begin{align*}
\frac{n!\beta\left(2\beta x\right)^{2\nu+1}}{\left(n+\nu+\frac{1}{2}\right)^{2\nu+3}\Gamma(2\nu+n+1)}e^{2\mu x-\frac{\mu n}{n+\nu+\frac{1}{2}}(x+y)}\mathcal{E}_{\alpha}(-\lambda_n t^{\alpha})\left|L^{2\nu}_{n}\left(\frac{2 \beta x}{n+\nu+\frac{1}{2}}\right) L^{2\nu}_{n}\left(\frac{2 \beta y}{n+\nu+\frac{1}{2}}\right)\right| \leq \frac{c_{x,y,\mu, \nu }}{n^3},	
\end{align*}	
and 
\begin{align*}
	\frac{\beta\left(2\beta x\right)^{2\nu+1}}{\left(n+\nu+\frac{1}{2}\right)^{2\nu+3}\Gamma(2\nu+1)}e^{2\mu x-\frac{\mu n}{n+\nu+\frac{1}{2}}x}\mathcal{E}_{\alpha}(-\lambda_n t^{\alpha})\left|L^{2\nu}_{n}\left(\frac{2 \beta x}{n+\nu+\frac{1}{2}}\right) \right| < \frac{c_{x,\mu, \nu}}{n^{2\nu+3}},
\end{align*}	
where $c_{x,y,\mu, \nu }$ and $c_{x,\mu, \nu }$  are some positive constants depending on $x, y, \mu$ and $\nu$. 

This verifies the convergence of the original series since $\nu>-1$.
\end{proof}
\begin{lemma}\label{Lemma2FBP}
Let $\nu \in \langle -1, +\infty \rangle \backslash \{-\frac{1}{2}\}, \, \mu \in \R$ and $\beta$ given by \eqref{beta_def}, with $\mathcal{E}_{\alpha}(\cdot)$ denoting Mittag-Leffler function and $M_{\chi, \mu}(\cdot)$ the Whittaker function. Then, for $x \geq 0$ and $y>0$
\begin{align*}
	\int\limits_{\mu^2/2}^{\infty}&\frac{1}{k(\lambda)}\left|\frac{\Gamma(\frac{1}{2}+\nu+i\beta/ k(\lambda))}{\Gamma(1+2\nu)}\right |^2 \mathcal{E}_{\alpha}\left(-\lambda t^{\alpha}\right)\left(\frac{x}{y}\right)^{\nu+\frac{1}{2}}e^{\mu(x-y)+\pi \beta /k(\lambda)} \\
	\times&\left|
 M_{i\beta/ k(\lambda),\, \nu}(-2ik(\lambda) y) M_{-i\beta/ k(\lambda), \, \nu}(2ik(\lambda) x)\right|\, d\lambda < +\infty,
\end{align*}
while for $x \geq 0$ and $-1 <\nu <1/2$,
\begin{equation*}
	\int\limits_{\mu^2/2}^{\infty}k(\lambda)^{\nu-1/2}\left|\frac{\Gamma(\frac{1}{2}+\nu+i\beta/ k(\lambda))}{\Gamma(1+2\nu)}\right |^2 \mathcal{E}_{\alpha}\left(-\lambda t^{\alpha}\right)x^{\nu+\frac{1}{2}}e^{\mu(x-y)+\pi \beta /k(\lambda)} 
	\left| M_{i\beta/ k(\lambda), \, \nu}(-2ik(\lambda) x)\right|\, d\lambda < +\infty.
\end{equation*}
\end{lemma}
\begin{proof}
Let 
\begin{small}\begin{equation*}
	h(\lambda)=	\frac{1}{k(\lambda)}\left|\frac{\Gamma(\frac{1}{2}+\nu+i\beta/ k(\lambda))}{\Gamma(1+2\nu)}\right |^2 \mathcal{E}_{\alpha}\left(-\lambda t^{\alpha}\right)\left(\frac{x}{y}\right)^{\nu+\frac{1}{2}}e^{\mu(x-y)+\pi \beta /k(\lambda)} \left|
	M_{i\beta/ k(\lambda),\, \nu}(-2ik(\lambda) y) M_{-i\beta/ k(\lambda), \, \nu}(2ik(\lambda) x)\right|.
\end{equation*}
\end{small}
According to \cite{NIST} (p. 335), the regularized Whittaker function $M_{\chi, \nu}(z)/\Gamma(1+2\nu)$ is an entire function in both parameters $\chi$ and $\nu$,
Therefore, the subintegral function $h(\lambda)$ is also entire.
Since 
\begin{equation*}
	\lim_{\lambda \to \mu^2/2} k(\lambda)=0, \quad \lim_{\lambda \to +\infty} k(\lambda)=+\infty
\end{equation*}
to complete the proof we must verify that $h(\lambda)$ is bounded by an integrable function in the neighborhood of the endpoints $\mu^2/2$ and $+\infty$. 
According to \cite{Buchholz} (p. 100, expression (21a)),
\begin{align*}
	\left| M_{i\beta/ k(\lambda),\, \nu}(-2ik(\lambda)y)\right | \thicksim c_{\beta, \nu, y} \cdot k(\lambda)^{\nu+\frac{1}{2}}, \quad \lambda \to \mu^2/2, \\
	\left| M_{-i\beta/ k(\lambda),\, \nu}(2ik(\lambda)x)\right | \thicksim c_{\beta, \nu, x} \cdot k(\lambda)^{\nu+\frac{1}{2}}, \quad \lambda \to \mu^2/2,
\end{align*}
where $c_{\beta, \nu, x}$ and $c_{\beta, \nu, y}$ are positive constants depending on $\beta, \nu$, $x$ and $y$.
On the other hand (see \cite{NIST} (p. 141)),
\begin{equation}
	\left|\Gamma\left(\frac{1}{2}+\nu+i\beta/ k(\lambda)\right)\right|^2 \thicksim 2\pi \left(\beta/k(\lambda)\right)^{2\nu} e^{-\pi \left|\beta\right|/k(\lambda)}, \quad \lambda \to \mu^2/2.
\end{equation}
Since the Mittag-Leffler function $\lambda \mapsto \mathcal{E}_{\alpha}(-\lambda t^{\alpha})$ is monotone decreasing and bounded (in fact, $\mathcal{E}_{\alpha}(0)=1$),
it follows that 
\begin{itemize}
\item For $\beta \geq 0$: $h(\lambda) \thicksim O(c), \,c \in \R, \, \lambda \to \mu^2/2$
\item For $\beta <0$: $h(\lambda) \thicksim O(e^{-\pi \left|\beta\right|/k(\lambda)}), \, \lambda \to \mu^2/2$.
\end{itemize}
In both cases, $h(\lambda)$ is dominated by an integrable function near the lower limit.

For the upper endpoint $\lambda \to +\infty $, the following asymptotic expansion holds for the Whittaker function (see \cite{Buchholz}, p. 92, expression (4)):
\begin{equation*}
	\left|  M_{i\beta/ k(\lambda),\, \nu}(-2ik(\lambda)y) \right|  \leq \left|\frac{e^{\frac{\pi}{2}\frac{\beta}{k(\lambda)}+i\left(\frac{1}{2}+\nu\right)}}{\Gamma\left(\frac{1}{2}+\nu+i \beta/k(\lambda)\right)}\right| \sum_{l=0}^{N}\frac{\left(\frac{\beta}{2 k(\lambda)^2y}\right)^l}{l!}\cdot \left(\prod_{r=1}^{l}\sin\rho^{(+)}_r \sin\rho^{(-)}_r\right)^{-1}, \lambda \to +\infty,
\end{equation*}
where
\begin{equation*}
	\tan{\rho^{(\pm)}_r}=\frac{\beta}{k(\lambda)(r-1/2\pm \nu)}, \quad \beta >0.
\end{equation*}
Using the identity
\begin{equation*}
	\left|\sin{\rho^{(\pm)}_r}\right|=\left|\frac{\tan{\rho^{(\pm)}_r}}{\sqrt{1+\tan{\rho^{(\pm)}_r}^2}}\right|	
\end{equation*}
we find
\begin{equation*}
\left|\left(\prod_{r=1}^{l}\sin\rho^{(+)}_r \sin\rho^{(-)}_r\right)^{-1}\right| \thicksim O(k(\lambda)^{2l}), \quad \lambda \to \infty.
\end{equation*}
Hence, since $k(\lambda) \to \infty$ as $\lambda \to \infty$,
\begin{equation*}
	\left|  M_{i\beta/ k(\lambda),\, \nu}(-2ik(\lambda)y) \right| \thicksim O(b_{\beta, \nu, y}), \quad \lambda \to \infty,
\end{equation*}
where $b_{\beta, \nu, x}$ is a positive constant depending on $\beta, \nu$ and $x$.  This asymptotic estimate is valid for $\beta > 0$, but using the symmetry of the Whittaker function in its first parameter,
\begin{equation*}
	M_{i\beta/ k(\lambda),\, \nu}(-2ik(\lambda)y)=e^{-i\pi(\nu+1/2)}M_{-i\beta/ k(\lambda),\, \nu}(2ik(\lambda)y)
\end{equation*}
we conclude that for $\beta \in \R$
\begin{equation*}
	\left|  M_{i\beta/ k(\lambda),\, \nu}(-2ik(\lambda)y) M_{-i\beta/ k(\lambda),\, \nu}(2ik(\lambda)x) \right| \thicksim O(b_{\beta, \nu, x,y}), \quad \lambda \to \infty.
\end{equation*}
From \cite{Mittag2011} (expression (6.5))
\begin{equation*}
	\mathcal{E}_{\alpha}(-\lambda t^{\alpha}) \thicksim \frac{1}{\Gamma(1-\alpha)\lambda t^{\alpha}}, \quad \lambda \to +\infty
\end{equation*}
we deduce
\begin{equation*}
	\left |h(\lambda)\right| \thicksim O(\lambda^{-\frac{3}{2}}), \quad \lambda \to +\infty.
\end{equation*}
This completes the proof for the first integral, while the proof for the second integral is omitted as it follows a similar reasoning.
\end{proof}
\begin{remark}
	The restriction on the parameter $\nu$ is essential to establish the convergence of the second integral. Consequently, for the fractional Bessel process to start at the initial point $y=0$, the parameter $\nu$ must lie within the range $\langle -1, 1/2 \rangle  \backslash \{-1/2\}$.
\end{remark}
Now we are ready to state and prove our main result.
\begin{theorem}\label{fracBesselprocess_trans_theorem}
	Spectral representation of the transition density of the fractional Bessel process \eqref{fracBesselprocess} is given by
	\begin{equation} \label{fracBesselprocess_trans}
		p_{\alpha}(x,t;y)=p_{\alpha, d}(x,t;y)+p_{\alpha, c}(x,t;y),
	\end{equation}
	where
	\[
	p_{\alpha, d}(x,t;y)=
	\begin{cases}
		\pi(x)+\sum\limits_{n=1}^{\infty}\frac{n!\beta\left(2\beta x\right)^{2\nu+1}}{\left(n+\nu+\frac{1}{2}\right)^{2\nu+3}\Gamma(2\nu+n+1)}e^{2\mu x-\frac{\mu n}{n+\nu+\frac{1}{2}}(x+y)} \\ \times \mathcal{E}_{\alpha}(-\lambda_n t^{\alpha})L^{2\nu}_{n}\left(\frac{2 \beta x}{n+\nu+\frac{1}{2}}\right)L^{2\nu}_{n}\left(\frac{2 \beta y}{n+\nu+\frac{1}{2}}\right) &,\, (\nu, \mu) \in \langle -1/2, +\infty \rangle \times \langle-\infty, 0 \rangle,\\
		\pi(x) &,\, (\nu, \mu) \in \langle -1, -1/2 \rangle \times \langle-\infty, 0 \rangle\\
		0 &,\, (\nu, \mu) \in \langle -1/2, +\infty \rangle \times \langle 0, +\infty \rangle\\
		\sum\limits_{n=1}^{\infty}\frac{n!\beta\left(2\beta x\right)^{2\nu+1}}{\left(n+\nu+\frac{1}{2}\right)^{2\nu+3}\Gamma(2\nu+n+1)}e^{2\mu x-\frac{\mu n}{n+\nu+\frac{1}{2}}(x+y)}\\ \times\mathcal{E}_{\alpha}(-\lambda_n t^{\alpha})L^{2\nu}_{n}\left(\frac{2 \beta x}{n+\nu+\frac{1}{2}}\right)L^{2\nu}_{n}\left(\frac{2 \beta y}{n+\nu+\frac{1}{2}}\right) &,\, (\nu, \mu) \in \langle -1, -1/2 \rangle \times \langle 0, +\infty \rangle.
	\end{cases}
	\]for $x \geq 0, y>0$, 
	
	and
		\[
	p_{\alpha, d}(x,t;0)=
	\begin{cases}
		\pi(x)+\sum\limits_{n=1}^{\infty}\frac{\beta\left(2\beta x\right)^{2\nu+1}}{\left(n+\nu+\frac{1}{2}\right)^{2\nu+3}\Gamma(2\nu+1)}e^{2\mu x-\frac{\mu n}{n+\nu+\frac{1}{2}}x} \\ \times \mathcal{E}_{\alpha}(-\lambda_n t^{\alpha})L^{2\nu}_{n}\left(\frac{2 \beta x}{n+\nu+\frac{1}{2}}\right) &,\, (\nu, \mu) \in \langle -1/2, 1/2 \rangle \times \langle-\infty, 0 \rangle,\\
		\pi(x) &,\, (\nu, \mu) \in \langle -1, -1/2 \rangle \times \langle-\infty, 0 \rangle\\
		0 &,\, (\nu, \mu) \in \langle -1/2, 1/2 \rangle \times \langle 0, +\infty \rangle\\
		\sum\limits_{n=1}^{\infty}\frac{\beta\left(2\beta x\right)^{2\nu+1}}{\left(n+\nu+\frac{1}{2}\right)^{2\nu+3}\Gamma(2\nu+1)}e^{2\mu x-\frac{\mu n}{n+\nu+\frac{1}{2}}x}\\ \times\mathcal{E}_{\alpha}(-\lambda_n t^{\alpha})L^{2\nu}_{n}\left(\frac{2 \beta x}{n+\nu+\frac{1}{2}}\right) &,\, (\nu, \mu) \in \langle -1, -1/2 \rangle \times \langle 0, +\infty \rangle.
	\end{cases}
	\]for $x \geq 0$,
	
	while
	\begin{align*}
		p_{\alpha, c}(x,t;y)=&\frac{1}{2\pi}\int\limits_{\mu^2/2}^{\infty}\frac{1}{k(\lambda)}\left|\frac{\Gamma(\frac{1}{2}+\nu+i\beta/ k(\lambda))}{\Gamma(1+2\nu)}\right |^2 \mathcal{E}_{\alpha}\left(-\lambda t^{\alpha}\right)\left(\frac{x}{y}\right)^{\nu+\frac{1}{2}}e^{\mu(x-y)+\pi \beta /k(\lambda)} \\
		 &\times M_{i\beta/ k(\lambda),\, \nu}(-2ik(\lambda) y) M_{-i\beta/ k(\lambda), \, \nu}(2ik(\lambda) x)\, d\lambda,
	\end{align*}
for $x \geq 0, y>0$

and

	\begin{align*}
	p_{\alpha, c}(x,t;0)=&\frac{1}{2\pi}\int\limits_{\mu^2/2}^{\infty}2^{\nu+1/2}k(\lambda)^{\nu - 1/2}e^{i\pi (\nu/2+1/4)}\left|\frac{\Gamma(\frac{1}{2}+\nu+i\beta/ k(\lambda))}{\Gamma(1+2\nu)}\right |^2 \mathcal{E}_{\alpha}\left(-\lambda t^{\alpha}\right)x^{\nu+\frac{1}{2}}e^{\mu x+\pi \beta /k(\lambda)} \\
	&\times M_{i\beta/ k(\lambda), \, \nu}(-2ik(\lambda) x)\, d\lambda,
\end{align*}
for $x \geq 0$ and $\nu \in \langle -1, 1/2 \rangle  \backslash \{-1/2\}$,
where $k(\lambda)=\sqrt{2\lambda-\mu^2}$, $\pi(\cdot)$ is given by \eqref{stat_dist_Bessel}, $L^{2\nu}_n(\cdot)$ is the generalized Laguerre polynomial, $\mathcal{E}_{\alpha}\left(\cdot \right)$ is the Mittag-Leffler function and $M_{\chi, \mu}(\cdot)$ is the Whittaker function.
\end{theorem}
\begin{proof}
The proof consists primarily of repeated applications of Fubini's theorem (and Fubini-Tonelli) theorem several times as follows. By the definition of the fractional Bessel process \eqref{fracBesselprocess} we have
	\begin{align} \label{FBP_tonelli}
		P(X_{\alpha}(t) \in B | X_{\alpha}(0)=y) &= \int_{0}^{\infty}P(X_{1}(\tau) \in B | X_{1}(0)=y)\,f_{t}(\tau)\,d\tau \nonumber \\
		&=\int_{0}^{\infty}\int_{B}^{}p_1(x,\tau;y)\,f_t(\tau)\,dx \, d\tau  \nonumber \\
		&=\int_{B}^{}\int_{0}^{\infty}(p_d(x,\tau;y)+p_c(x,\tau;y))\, f_t(\tau)\,d\tau \,dx,
	\end{align}
where the second equality follows from the definition of the transition density of the non-fractional Bessel process, and the third equality uses the Fubini–Tonelli theorem to justify the exchange of the order of integration.

Let us first consider the case $y>0$. In the next step, we substitute the spectral representation of the Bessel process from \eqref{Bessel_trans_dens}, yielding:
\begin{align}\label{FBP_proof1}
	\int_{0}^{\infty}p_c(x,\tau;y)\, f_t(\tau)\,d\tau=\int_{0}^{\infty}&\frac{1}{2\pi}\int\limits_{\mu^2/2}^{\infty}\frac{1}{k(\lambda)}\left|\frac{\Gamma(\frac{1}{2}+\nu+i\beta/ k(\lambda))}{\Gamma(1+2\nu)}\right |^2 e^{-\lambda \tau}\left(\frac{x}{y}\right)^{\nu+\frac{1}{2}}e^{\mu(x-y)+\pi \beta /k(\lambda)} \nonumber \\ 
	&\times M_{i\beta/ k(\lambda),\, \nu}(-2ik(\lambda) y) M_{-i\beta/ k(\lambda), \, \nu}(2ik(\lambda) x) f_t(\tau) \, d\lambda d\tau.
\end{align}
In the next step, we change the order of integration using Fubini's theorem (rather than Tonelli’s), since the integrand is not necessarily nonnegative. To justify this step, we verify the absolute integrability of the integrand:
\begin{align*}
	\int_{0}^{\infty}&\frac{1}{2\pi}\int\limits_{\mu^2/2}^{\infty}\frac{1}{k(\lambda)}\left|\frac{\Gamma(\frac{1}{2}+\nu+i\beta/ k(\lambda))}{\Gamma(1+2\nu)}\right |^2 e^{-\lambda \tau}\left(\frac{x}{y}\right)^{\nu+\frac{1}{2}}e^{\mu(x-y)+\pi \beta /k(\lambda)} \\
	&\times \left|M_{i\beta/ k(\lambda),\, \nu}(-2ik(\lambda) y) M_{-i\beta/ k(\lambda), \, \nu}(2ik(\lambda) x)\right| f_t(\tau) \, d\lambda \, d\tau= \\
	=\frac{1}{2\pi}\int\limits_{\mu^2/2}^{\infty}\frac{1}{k(\lambda)}&\left|\frac{\Gamma(\frac{1}{2}+\nu+i\beta/ k(\lambda))}{\Gamma(1+2\nu)}\right |^2 \left(\frac{x}{y}\right)^{\nu+\frac{1}{2}}e^{\mu(x-y)+\pi \beta /k(\lambda)} \\
	&\times \left|M_{i\beta/ k(\lambda),\, \nu}(-2ik(\lambda) y) M_{-i\beta/ k(\lambda), \, \nu}(2ik(\lambda) x)\right| \int_{0}^{\infty} e^{-\lambda \tau} f_t(\tau) \,  d\tau \, d\lambda= \\
	=\frac{1}{2\pi}\int\limits_{\mu^2/2}^{\infty}\frac{1}{k(\lambda)}&\left|\frac{\Gamma(\frac{1}{2}+\nu+i\beta/ k(\lambda))}{\Gamma(1+2\nu)}\right |^2 \left(\frac{x}{y}\right)^{\nu+\frac{1}{2}}e^{\mu(x-y)+\pi \beta /k(\lambda)} \\
	&\times \left|M_{i\beta/ k(\lambda),\, \nu}(-2ik(\lambda) y) M_{-i\beta/ k(\lambda), \, \nu}(2ik(\lambda) x)\right| \mathcal{E}\left(-\lambda t^{\alpha}\right) \,  d\tau \, d\lambda< +\infty,
\end{align*}
where in the first equality we applied the Fubini–Tonelli theorem, in the second we used \eqref{Mittag}, and the finiteness of the integral follows from a direct application of Lemma \ref{Lemma2FBP}.

Now, \eqref{FBP_proof1} simplifies  to
\begin{align*}
	\int_{0}^{\infty}p_c(x,\tau;y)\, f_t(\tau)\,d\tau=\frac{1}{2\pi}\int\limits_{\mu^2/2}^{\infty}&\int_{0}^{\infty}\frac{1}{k(\lambda)}\left|\frac{\Gamma(\frac{1}{2}+\nu+i\beta/ k(\lambda))}{\Gamma(1+2\nu)}\right |^2 e^{-\lambda \tau}\left(\frac{x}{y}\right)^{\nu+\frac{1}{2}}e^{\mu(x-y)+\pi \beta /k(\lambda)} \\ 
	&\times M_{i\beta/ k(\lambda),\, \nu}(-2ik(\lambda) y) M_{-i\beta/ k(\lambda), \, \nu}(2ik(\lambda) x) f_t(\tau) \, d\lambda d\tau \\
	=\frac{1}{2\pi}\int\limits_{\mu^2/2}^{\infty}\frac{1}{k(\lambda)}&\left|\frac{\Gamma(\frac{1}{2}+\nu+i\beta/ k(\lambda))}{\Gamma(1+2\nu)}\right |^2 \left(\frac{x}{y}\right)^{\nu+\frac{1}{2}}e^{\mu(x-y)+\pi \beta /k(\lambda)} \\ 
	&\times M_{i\beta/ k(\lambda),\, \nu}(-2ik(\lambda) y) M_{-i\beta/ k(\lambda), \, \nu}(2ik(\lambda) x) \int_{0}^{\infty}e^{-\lambda \tau} f_t(\tau) \, d\lambda d\tau \\
	=\frac{1}{2\pi}\int\limits_{\mu^2/2}^{\infty}\frac{1}{k(\lambda)}&\left|\frac{\Gamma(\frac{1}{2}+\nu+i\beta/ k(\lambda))}{\Gamma(1+2\nu)}\right |^2 \mathcal{E}_{\alpha}\left(-\lambda t^{\alpha}\right) \left(\frac{x}{y}\right)^{\nu+\frac{1}{2}}e^{\mu(x-y)+\pi \beta /k(\lambda)} \\ 
	&\times M_{i\beta/ k(\lambda),\, \nu}(-2ik(\lambda) y) M_{-i\beta/ k(\lambda), \, \nu}(2ik(\lambda) x) \, d\lambda. 
\end{align*}
This confirms that the continuous part of the spectral representation is given by
\begin{align}\label{FBP_cont_part}
p_{\alpha, c}(x,t; y)=\frac{1}{2\pi}\int\limits_{\mu^2/2}^{\infty}\frac{1}{k(\lambda)}&\left|\frac{\Gamma(\frac{1}{2}+\nu+i\beta/ k(\lambda))}{\Gamma(1+2\nu)}\right |^2 \mathcal{E}_{\alpha}\left(-\lambda t^{\alpha}\right) \left(\frac{x}{y}\right)^{\nu+\frac{1}{2}}e^{\mu(x-y)+\pi \beta /k(\lambda)} \nonumber \\ 
&\times M_{i\beta/ k(\lambda),\, \nu}(-2ik(\lambda) y) M_{-i\beta/ k(\lambda), \, \nu}(2ik(\lambda) x) \, d\lambda. 
\end{align}
The discrete part of the spectral representation is handled based on the values of the parameters $\nu$ and $\mu$:
\begin{itemize}
	\item If $(\nu, \mu) \in \langle -1, -1/2 \rangle \times \langle-\infty, 0 \rangle$, then $p_d(x,t;y)=\pi(x)$, so
	\begin{equation}\label{FBP_disc_part1}
		p_{\alpha, d}(x,t;y)=\int_{0}^{\infty}p_d(x,t;y)f_t(\tau)\,d\tau=\pi(x)\int_{0}^{\infty}f_t(\tau)\,d\tau=\pi(x),
	\end{equation}
where the last equality follows from the fact that $f_t(\cdot)$ is a probability density function.
	\item If $(\nu, \mu) \in \langle -1/2, +\infty \rangle \times \langle 0, +\infty \rangle$, then $p_d(x,t;y)=0$, and thus
	\begin{equation}\label{FBP_disc_part2}
		p_{\alpha, d}(x,t;y)=0.
	\end{equation}	
	\item If $(\nu, \mu) \in \langle -1, -1/2 \rangle \times \langle 0, +\infty \rangle$, we apply the discrete part of the spectral representation of the Bessel process. Then, using Fubini's theorem, we obtain:
	\begin{align}\label{FBP_disc_part3}
		p_{\alpha, d}(x,t;y)=\int_{0}^{\infty}p_d(x,\tau;y)f_t(\tau)\, d\tau=\sum\limits_{n=1}^{\infty}\frac{n!\beta\left(2\beta x\right)^{2\nu+1}}{\left(n+\nu+\frac{1}{2}\right)^{2\nu+3}\Gamma(2\nu+n+1)}e^{2\mu x-\frac{\mu n}{n+\nu+\frac{1}{2}}(x+y)}\nonumber\\ \times L^{2\nu}_{n}\left(\frac{2 \beta x}{n+\nu+\frac{1}{2}}\right)L^{2\nu}_{n}\left(\frac{2 \beta y}{n+\nu+\frac{1}{2}}\right)\int_{0}^{\infty} e^{-\lambda_n t}f_t(\tau)d\tau \nonumber\\
		=\sum\limits_{n=1}^{\infty}\frac{n!\beta\left(2\beta x\right)^{2\nu+1}}{\left(n+\nu+\frac{1}{2}\right)^{2\nu+3}\Gamma(2\nu+n+1)}e^{2\mu x-\frac{\mu n}{n+\nu+\frac{1}{2}}(x+y)}\nonumber\\ \times\mathcal{E}_{\alpha}(-\lambda_n t^{\alpha})L^{2\nu}_{n}\left(\frac{2 \beta x}{n+\nu+\frac{1}{2}}\right)L^{2\nu}_{n}\left(\frac{2 \beta y}{n+\nu+\frac{1}{2}}\right),
	\end{align}
	where the use of Fubini's theorem in the second equality is justified by Lemma \ref{Lemma1FBP}.
		\item If $(\nu, \mu) \in \langle -1/2, +\infty \rangle \times \langle-\infty, 0 \rangle$, then by Lemma \ref{Lemma1FBP}, we similarly obtain:
	\begin{align}\label{FBP_disc_part4}
		p_{\alpha, d}(x,t;y)=\pi(x)+\sum\limits_{n=1}^{\infty}\frac{n!\beta\left(2\beta x\right)^{2\nu+1}}{\left(n+\nu+\frac{1}{2}\right)^{2\nu+3}\Gamma(2\nu+n+1)}e^{2\mu x-\frac{\mu n}{n+\nu+\frac{1}{2}}(x+y)}\nonumber\\ \times\mathcal{E}_{\alpha}(-\lambda_n t^{\alpha})L^{2\nu}_{n}\left(\frac{2 \beta x}{n+\nu+\frac{1}{2}}\right)L^{2\nu}_{n}\left(\frac{2 \beta y}{n+\nu+\frac{1}{2}}\right).
	\end{align}
\end{itemize}
	Finally, combining \eqref{FBP_cont_part}, \eqref{FBP_disc_part1}, \eqref{FBP_disc_part2}, \eqref{FBP_disc_part3} and \eqref{FBP_disc_part4} yields the full spectral representation of the transition density of the fractional Bessel process, as given in \eqref{fracBesselprocess_trans}. Considering  Lemmas \ref{Lemma1FBP} and  \ref{Lemma2FBP} the case $y=0$ is addressed in a similar fashion and therefore we skip this part of the proof.

\end{proof}
\begin{remark}
Notice that for $\alpha=1$, the fractional case reduces to the non-fractional case. In particular, since $E(t)=t \text{ a.s.}$, it follows that
\begin{equation*}
	X_{\alpha}(t)=X_1(t)=X(E(t))=X(t) \text{ a.s.}
\end{equation*} 
Moreover, Mittag-Leffler function reduces to the exponential function: 
\begin{equation*}
	\mathcal{E}_{1}\left(-\lambda t\right)=e^{-\lambda t}
\end{equation*}
and therefore the transition density  of fractional diffusion $p_{\alpha}(x,t;y)$ reduces to the transition density of the non-fractional diffusion $p(x,t;y)$. This observation explains why the defined fractional model can be seen as extension of the non-fractional model.
\end{remark}

\subsection{Fractional Cauchy problem for fractional Bessel process}
The transition density of a diffusion process, which is a solution of SDE
\begin{equation*}
	dX(t)=\mu (X(t))dt + \sigma(X(t))dB(t)
\end{equation*} 
solves Kolmogorov backward partial differential equation involving the corresponding infinitesimal generator:
        \begin{equation}\label{KBPE}
	\frac{\partial p(x,t;y)}{\partial t}= \mu(y)\frac{\partial p(x,t;y)}{\partial y}+\frac{\sigma^2(y)}{2}\frac{\partial^2 p(x,t;y)}{\partial y^2}=\mathcal{G}p
\end{equation}
with the point-source initial condition $p(x,0;y)=\delta(x-y)$, where $\delta$ denotes the Dirac delta function. 

Furthermore, under standard conditions, the family of operators $(T_t, t \geq 0)$, 
\begin{equation*}
	T_t f(y)= \mathbb{E}[f(X(t)) \,|\, X(0)=y]
\end{equation*}
forms a strongly continuous bounded ($C_0$) semigroup on the space of bounded continuous functions $f$ on $\left[0, \infty \right>$ vanishing at infinity (see e.g. \cite{LeonenkoMeerschaertSikorskii_FPD_2013}).
This fact, together with [Theorem 3.1.9, \cite{Arendt}] ensures that the strong (classical) solution of Cauchy problem
\begin{equation}\label{CauchyProblem}
	\frac{\partial u(y,t)}{\partial t}= \mathcal{G} u(y,t), \quad u(y,0)=f(y).
\end{equation}
is given by
\begin{equation*}
u(t;y)=T_t f(y)=\mathbb{E}[f(X(t)) \,|\, X(0)=y]=\int_{0}^{\infty}f(x)p(x,t;y)dx.
\end{equation*}
Moreover, if spectral representation of transition density of a diffusion is available, one can write explicit strong solution in spectral form. 

Fractional diffusions have analogues of equations \eqref{KBPE} and \eqref{CauchyProblem}. Replacing the time derivative on the left-hand side of \eqref{CauchyProblem} with the Caputo fractional derivative \( \frac{\partial^{\alpha}}{\partial t^{\alpha}} \), one obtains the fractional Cauchy problem:
\begin{equation}\label{fracCauchyProblem}
	\frac{\partial^\alpha u(y,t)}{\partial t^\alpha}= \mathcal{G} u(y,t), \quad u(y,0)=f(y).
\end{equation}
The Caputo fractional derivative of order $\alpha \in \langle 0, 1 \rangle$ is defined by
$$\frac{d^{\alpha}f(x)}{dx^{\alpha}} = \frac{1}{\Gamma(1-\alpha)} \int\limits_{0}^{\infty} \frac{d}{dx} f(x-y) y^{-\alpha} \, dy.$$
It turns out that explicit strong solutions to \eqref{CauchyProblem} and \eqref{fracCauchyProblem} are connected via their corresponding stochastic solutions, provided that spectral representations of their transition densities are available.
\begin{remark}
For a general treatment on fractional calculus we refer the reader to \cite{Kilbas2006} and \cite{Podlubny}. We emphasize that there exist other definitions of fractional derivatives such as Riemann-Liouville,  Grunwald-Letnikov, Hadamard, and many others, but Caputo-type derivative is of particular interest in probability theory. The primary reason is that it enables linking time-changed stochastic processes with solutions of certain partial differential equations via fractional calculus. This connection explains why the stochastic process defined in \eqref{fracBesselprocess} is referred to as a fractional diffusion.
\end{remark}
We now return to the Bessel process. As we have already obtained the explicit spectral representation of the transition density for the fractional Bessel process (see Theorem \ref{fracBesselprocess_trans_theorem}), we are equipped with the necessary tools to solve the fractional Cauchy problem \eqref{fracCauchyProblem}. For the Bessel process, this problem reduces to:
\begin{equation}\label{fracBesselCauchyProblem}
		\frac{\partial^\alpha u(y,t)}{\partial t^\alpha}= \left(\frac{\nu+\frac{1}{2}}{y}+\mu\right)\frac{\partial u(y,t)}{\partial y}+\frac{1}{2}\frac{\partial^2 u(y,t)}{\partial y^2} , \quad u(y,0)=f(y).
\end{equation}
\begin{theorem}\label{StrongSolutionFracProbBessel}
If the initial condition function $f$ belongs to the domain of the generator $\mathcal{G}$ as defined in \eqref{BesselDomain}, then a strong solution to the fractional Cauchy problem \eqref{fracBesselCauchyProblem} is given by
	\begin{equation}\label{fbackwardsolution}
		u_{\alpha}(t;y)=\int _0^{\infty} p_\alpha (x,t;y)f(x)dx
	\end{equation}
	where the transition density $p_\alpha $ is given by equation \eqref{fracBesselprocess_trans}.
\end{theorem}
\begin{proof}
According to \cite{linetsky_2004}, the family of operators $(T_t, \, t \geq 0)$, where
\begin{equation*}
	T_t f(y)=\mathbb{E}[f(X(t))\,|\,X(0)=y]
\end{equation*}
forms a strongly continuous, bounded semigroup on the space of bounded continuous functions on $\left[0, \infty\right>$, and therefore also on the subspace of bounded continuous functions on $\left[0, \infty \right>$ vanishing at infinity.

Then, [Theorem 3.1.9, \cite{Arendt}] yields that $T_t f(y)=E[f(X(t))\,|\,X(0)=y]$ is a strong solution for the Cauchy problem
\begin{equation*}
	\frac{\partial u(y,t)}{\partial t}= \left(\frac{\nu+\frac{1}{2}}{y}+\mu\right)\frac{\partial u(y,t)}{\partial y}+\frac{1}{2}\frac{\partial^2 u(y,t)}{\partial y^2} , \quad u(y,0)=f(y).	
\end{equation*}
Next, according to [Theorem 3.1, \cite{BaeumerMeerschaert_2001}], for any $g$ in the domain of the generator $\mathcal G$,
\begin{equation}
S_t g (y) = \int\limits_{0}^{\infty}T_u g(y) \,f_t(u)\,du,
\end{equation}
solves the fractional Cauchy problem \eqref{fracBesselCauchyProblem}, where $f_t$ is the density of the inverse stable subordinator $E_t$, given by \eqref{Mittag}.
Finally, since $E(0)=0$ almost surely, and rewriting the expression in \eqref{fbackwardsolution}, we have:
\begin{eqnarray*}
\int_{0}^{\infty} p_{\alpha}(x,t;y)\,f(x)\,dx &=& \mathbb{E}[f(X_{\alpha}(t))\,|\,X_{\alpha}(0)=y]\\
	&=& \int\limits_{0}^{\infty}\mathbb{E}[f(X(u))\,|\,X(0)=y]\, f_t(u)\,du\\
	&=& \int\limits_{0}^{\infty}T_{u}\, f(y)\, f_t(u)\,du=S_t f(y). \\
\end{eqnarray*}
Therefore, for any $f \in D(\mathcal{G})$, \eqref{fbackwardsolution} is a strong solution to \eqref{fracBesselCauchyProblem}.
\end{proof}
As a direct consequence, we can express explicit strong solutions in spectral representation form, depending on the values of the model parameters.
\begin{cor}
For any $f \in D(\mathcal{G})$, a strong solution of fractional Cauchy problem \eqref{fracBesselCauchyProblem} can be written in the form:
%
%
%
\begin{align*}	u_{\alpha}&(y,t)=\frac{1}{2\pi}\int_{0}^{\infty}f(x)\int\limits_{\mu^2/2}^{\infty}\frac{1}{k(\lambda)}\left|\frac{\Gamma(\frac{1}{2}+\nu+i\beta/ k(\lambda))}{\Gamma(1+2\nu)}\right |^2 \mathcal{E}_{\alpha}\left(-\lambda t^{\alpha}\right)\left(\frac{x}{y}\right)^{\nu+\frac{1}{2}}e^{\mu(x-y)+\pi \beta /k(\lambda)} \\
	&\times  M_{i\beta/ k(\lambda),\, \nu}(-2ik(\lambda) y) M_{-i\beta/ k(\lambda), \, \nu}(2ik(\lambda) x) d \lambda\, dx \,+ \\
	&+\begin{cases}
		\int_{0}^{\infty}f(x)\pi(x)dx+\sum\limits_{n=1}^{\infty}\frac{n!\beta}{\left(n+\nu+\frac{1}{2}\right)^{2\nu+3}\Gamma(2\nu+n+1)}e^{-\frac{\mu n}{n+\nu+\frac{1}{2}}y}\mathcal{E}_{\alpha}(-\lambda_n t^{\alpha}) \\ \times L^{2\nu}_{n}\left(\frac{2 \beta y}{n+\nu+\frac{1}{2}}\right)\int_{0}^{\infty}\left(2\beta x\right)^{2\nu+1}L^{2\nu}_{n}\left(\frac{2 \beta x}{n+\nu+\frac{1}{2}}\right) e^{2\mu x-\frac{\mu n x}{n+\nu+\frac{1}{2}}} f(x) dx &,\, (\nu, \mu) \in \langle -1/2, +\infty \rangle \times \langle-\infty, 0 \rangle\\
		\int_{0}^{\infty}f(x)\pi(x)dx &,\, (\nu, \mu) \in \langle -1, -1/2 \rangle \times \langle-\infty, 0 \rangle\\
		0 &,\, (\nu, \mu) \in \langle -1/2, +\infty \rangle \times \langle 0, +\infty \rangle\\
		\sum\limits_{n=1}^{\infty}\frac{n!\beta}{\left(n+\nu+\frac{1}{2}\right)^{2\nu+3}\Gamma(2\nu+n+1)}e^{-\frac{\mu n}{n+\nu+\frac{1}{2}}y}\mathcal{E}_{\alpha}(-\lambda_n t^{\alpha}) \\ \times L^{2\nu}_{n}\left(\frac{2 \beta y}{n+\nu+\frac{1}{2}}\right)\int_{0}^{\infty}\left(2\beta x\right)^{2\nu+1}L^{2\nu}_{n}\left(\frac{2 \beta x}{n+\nu+\frac{1}{2}}\right) e^{2\mu x-\frac{\mu n x}{n+\nu+\frac{1}{2}}} f(x) dx &,\, (\nu, \mu) \in \langle -1, -1/2 \rangle \times \langle 0, +\infty \rangle
	\end{cases}
\end{align*}
for $y>0$, and 
\begin{align*}	u_{\alpha}&(0,t)=\frac{1}{2\pi}\int_{0}^{\infty}f(x)\int\limits_{\mu^2/2}^{\infty}2^{\nu+1/2}k(\lambda)^{\nu - 1/2}e^{i\pi (\nu/2+1/4)}\left|\frac{\Gamma(\frac{1}{2}+\nu+i\beta/ k(\lambda))}{\Gamma(1+2\nu)}\right |^2 \mathcal{E}_{\alpha}\left(-\lambda t^{\alpha}\right) \\
	& \times    x^{\nu+\frac{1}{2}}e^{\mu x+\pi \beta /k(\lambda)}M_{i\beta/ k(\lambda), \, \nu}(-2ik(\lambda) x) d \lambda\ dx + \\
	&+\begin{cases}
		\int_{0}^{\infty}f(x)\pi(x)dx+\sum\limits_{n=1}^{\infty}\frac{\beta}{\left(n+\nu+\frac{1}{2}\right)^{2\nu+3}\Gamma(2\nu+1)}e^{-\frac{\mu n}{n+\nu+\frac{1}{2}}y}\mathcal{E}_{\alpha}(-\lambda_n t^{\alpha}) \\ \times \int_{0}^{\infty}\left(2\beta x\right)^{2\nu+1}L^{2\nu}_{n}\left(\frac{2 \beta x}{n+\nu+\frac{1}{2}}\right) e^{2\mu x-\frac{\mu n x}{n+\nu+\frac{1}{2}}} f(x) dx &,\, (\nu, \mu) \in \langle -1/2, +\infty \rangle \times \langle-\infty, 0 \rangle\\
		\int_{0}^{\infty}f(x)\pi(x)dx &,\, (\nu, \mu) \in \langle -1, -1/2 \rangle \times \langle-\infty, 0 \rangle\\
		0 &,\, (\nu, \mu) \in \langle -1/2, +\infty \rangle \times \langle 0, +\infty \rangle\\
		\sum\limits_{n=1}^{\infty}\frac{\beta}{\left(n+\nu+\frac{1}{2}\right)^{2\nu+3}\Gamma(2\nu+1)}e^{-\frac{\mu n}{n+\nu+\frac{1}{2}}y}\mathcal{E}_{\alpha}(-\lambda_n t^{\alpha}) \\ \times \int_{0}^{\infty}\left(2\beta x\right)^{2\nu+1}L^{2\nu}_{n}\left(\frac{2 \beta x}{n+\nu+\frac{1}{2}}\right) e^{2\mu x-\frac{\mu n x}{n+\nu+\frac{1}{2}}} f(x) dx &,\, (\nu, \mu) \in \langle -1, -1/2 \rangle \times \langle 0, +\infty \rangle
	\end{cases}
\end{align*}
for $y=0$.
\end{cor}
\begin{proof}
This result follows by a direct application of Theorem \ref{fracBesselprocess_trans_theorem}  and Theorem \ref{StrongSolutionFracProbBessel}.
\end{proof}
\subsection{Stationary distribution of fractional Bessel process}
Here, we prove that the stationary distribution of the fractional Bessel process coincides with the stationary distribution of the non-fractional Bessel process (provided it exists). In other words, both the fractional and non-fractional Bessel processes share the same stationary distribution.

\begin{theorem}
Let $\left(X_{\alpha}(t), \, t \geq 0\right)$ be the fractional Bessel process defined via \eqref{fracBesselprocess} and let $p_{\alpha}(x,t)$ denote the probability density function of $X_{\alpha}(t)$. If $(\nu, \mu ) \in \langle -1, \infty \rangle \setminus \left\{ -1/2 \right\}  \times \langle -\infty, 0\rangle$, then for any initial probability density function $f$ of $X_{\alpha}(0)$ vanishing at infinity, we have
\begin{equation*}
p_{\alpha}(x,t) \to \pi(x), \,\text{ as }\,  t \to \infty,
\end{equation*}
where $\pi(\cdot)$ is a gamma stationary distribution \eqref{stat_dist_Bessel}.
\end{theorem}
\begin{proof}
It is sufficient to prove the convergence of the transition density, i.e., to show that
\begin{equation}\label{trans_limit}
	p_{\alpha}(x,t;y) \to \pi(x), \,\text{ as }\, t \to \infty
\end{equation}
for fixed $x$ and $y$.
To see this, recall that
\begin{equation*}
	p_{\alpha}(x,t)=\int_{0}^{\infty} p_{\alpha}(x,t;y) f(y)dy. 
\end{equation*}
Assuming \eqref{trans_limit}, we obtain
\begin{align*}
	\lim_{t \to \infty} p_{\alpha}(x,t)&=\lim_{t \to \infty}\int_{0}^{\infty} p_{\alpha}(x,t;y) f(y)dy  \\
	&=\int_{0}^{\infty} \lim_{t \to \infty} p_{\alpha}(x,t;y) f(y)dy \\
	&=\pi(x)\int_{0}^{\infty}f(y)dy \\
	&= \pi(x).
\end{align*}
The second equality follows from the dominated convergence theorem, which applies because $f(y)$ and $p_{\alpha}(x, t; y)$ are density functions and $f$ vanishes at infinity.
Thus, it remains to prove \eqref{trans_limit}. Without loss of generality, we consider the case $y>0$. Recalling the spectral representation of transition density of the fractional Bessel process (see Theorem \ref{fracBesselprocess_trans_theorem}), we need to show that the continuous part vanishes, while discrete part converges to $\pi(x)$ as $t$ approaches infinity.

We begin by proving that the continuous part of the spectral representation, \( p_{\alpha, c}(x, t; y) \), vanishes as \( t \to \infty \). According to [Theorem 4, \cite{Simon_Mittag}], the Mittag-Leffler function satisfies the following bounds:
\begin{equation}\label{Mittag_Leffler_unif_bounds}
0 <\frac{1}{1+\Gamma(1-\alpha)x}\leq \mathcal{E}_{\alpha}(-x) \leq \frac{1}{1+\Gamma(1+\alpha)^{-1}x}
\end{equation}
for $x  \geq 0$.
In particular, for any constant $c \geq \left(\mu^2/2 \cdot t^{\alpha}\right)^{-1}+\Gamma(1-\alpha)$ we have:
\begin{equation*}
	0 <\frac{1}{c \lambda t^\alpha}\leq \frac{1}{1+\Gamma(1-\alpha)\lambda t^{\alpha}} \leq \mathcal{E}_{\alpha}(-\lambda t^{\alpha}) \leq \frac{1}{1+\Gamma(1+\alpha)^{-1}\lambda t^{\alpha}} \leq  \frac{1}{\Gamma(1+\alpha)^{-1}\lambda t^{\alpha}}.
\end{equation*} 
Now
\begin{align*}
	 	 &\frac{c}{2\pi}\frac{1}{t^{\alpha}}\int\limits_{\mu^2/2}^{\infty}\frac{1}{k(\lambda)}\left|\frac{\Gamma(\frac{1}{2}+\nu+i\beta/ k(\lambda))}{\Gamma(1+2\nu)}\right |^2 \frac{1}{\lambda}\left(\frac{x}{y}\right)^{\nu+\frac{1}{2}}e^{\mu(x-y)+\pi \beta /k(\lambda)} \\
	 &\times M_{i\beta/ k(\lambda),\, \nu}(-2ik(\lambda) y) M_{-i\beta/ k(\lambda), \, \nu}(2ik(\lambda) x)\, d\lambda \\
	&\leq p_{\alpha, c}(x,t;y) \\ 
	 &\leq \frac{1}{2\pi}\frac{1}{t^{\alpha}}\int\limits_{\mu^2/2}^{\infty}\frac{1}{k(\lambda)}\left|\frac{\Gamma(\frac{1}{2}+\nu+i\beta/ k(\lambda))}{\Gamma(1+2\nu)}\right |^2 \frac{\Gamma(1+\alpha)}{\lambda}\left(\frac{x}{y}\right)^{\nu+\frac{1}{2}}e^{\mu(x-y)+\pi \beta /k(\lambda)} \\
	&\times M_{i\beta/ k(\lambda),\, \nu}(-2ik(\lambda) y) M_{-i\beta/ k(\lambda), \, \nu}(2ik(\lambda) x)\, d\lambda
\end{align*} 
for fixed $x$ and $y$. Since the integrals involved are finite (by Lemma \ref{Lemma2FBP}), we conclude 
\begin{equation*}
p_{\alpha, c}(x,t;y) \to 0, \text{ as } t \to \infty.
\end{equation*}
As for the discrete part:
\begin{itemize}
	\item  If $(\nu, \mu ) \in \langle-1, -1/2 \rangle \times \langle -\infty, 0\rangle$, the result is immediate, since $p_{\alpha, d}(x,t;y)=\pi(x)$. 
	\item  If $(\nu, \mu ) \in \langle-1/2, +\infty \rangle \times \langle -\infty, 0\rangle$, then
\begin{align*}
\lim_{t \to \infty}p_{\alpha, d}(x,t;y)&= \pi(x)+\lim_{t \to \infty}\sum\limits_{n=1}^{\infty}\frac{n!\beta\left(2\beta x\right)^{2\nu+1}}{\left(n+\nu+\frac{1}{2}\right)^{2\nu+3}\Gamma(2\nu+n+1)}e^{2\mu x-\frac{\mu n}{n+\nu+\frac{1}{2}}(x+y)} \\ 
&\times \mathcal{E}_{\alpha}(-\lambda_n t^{\alpha})L^{2\nu}_{n}\left(\frac{2 \beta x}{n+\nu+\frac{1}{2}}\right)L^{2\nu}_{n}\left(\frac{2 \beta y}{n+\nu+\frac{1}{2}}\right) \\
&= \pi(x)+\sum\limits_{n=1}^{\infty}\frac{n!\beta\left(2\beta x\right)^{2\nu+1}}{\left(n+\nu+\frac{1}{2}\right)^{2\nu+3}\Gamma(2\nu+n+1)}e^{2\mu x-\frac{\mu n}{n+\nu+\frac{1}{2}}(x+y)} \\ 
&\times \left(\lim_{t \to \infty}\mathcal{E}_{\alpha}(-\lambda_n t^{\alpha})\right)L^{2\nu}_{n}\left(\frac{2 \beta x}{n+\nu+\frac{1}{2}}\right)L^{2\nu}_{n}\left(\frac{2 \beta y}{n+\nu+\frac{1}{2}}\right) \\
&=\pi(x).
\end{align*}
The use of dominated convergence theorem in the second equality is justified by \eqref{Mittag_Leffler_unif_bounds} and arguments similar to those in the proof of Lemma \ref{Lemma1FBP}.
Finally, for fixed $x$ and $y$, we conclude:
\begin{equation*}
\lim_{t \to \infty}p_{\alpha}(x,t;y)=\lim_{t \to \infty}p_{\alpha,d}(x,t;y)+\lim_{t \to \infty}p_{\alpha,c}(x,t;y)=\pi(x),
\end{equation*}
which completes the proof.
\end{itemize}
\end{proof}
\subsection{Correlation Structure of the fractional Bessel process}

In this subsection, we provide an explicit formula for the correlation structure of both the fractional Bessel process and its eigenfunction-based transformation. Moreover, we show that both the fractional Bessel process and the transformed process exhibit the long-range dependence property, i.e., their correlation function falls of like a power law with exponent $0 < \alpha <1$.

We assume that the (non-fractional) Bessel process is stationary, i.e., that a stationary gamma distribution \( \pi \) exists (which requires \( (\nu, \mu) \in \langle -1, \infty \rangle \setminus \left\{ -1/2 \right\} \times \langle -\infty, 0 \rangle \)) and that \( X(0) \sim \pi \). Although the fractional Bessel process is non-stationary, we say it is in the steady state if it starts from the stationary distribution with density \( \pi \).

We begin by showing that a space transformation of the fractional Bessel process via eigenfunctions yields a long-range dependent stochastic process.

\begin{theorem}\label{Corr_structure_frac_transfBessel}
	Let $m, n \in \N$ and $t\geq s >0$. Then
\begin{align*}
	\mathrm{Corr}[\Psi_m(X_{\alpha}(t)), \Psi_n(X_{\alpha}(s))] 
	&= \mathrm{Corr}[\Psi_m(X_{\alpha}(0)), \Psi_n(X_{\alpha}(0))] \\
	&\quad \times \left( \mathcal{E}_{\alpha}(-\lambda_m t^{\alpha}) 
	+ \frac{\lambda_m t^{\alpha}}{\Gamma(\alpha)} 
	\int_{0}^{s/t} \frac{\mathcal{E}_{\alpha}(-\lambda_m t^{\alpha}(1 - z)^{\alpha})}{z^{1 - \alpha}}\, dz \right).
\end{align*}
\end{theorem}
\begin{proof}
According to \cite{kessler1999estimating}, for any eigenfunction–eigenvalue pair $(f, \lambda)$ of the diffusion generator $\mathcal{G}$, i.e., for $f \in D(\mathcal{G})$ such that
\begin{equation*}
	\mathcal{G}f=-\lambda f
\end{equation*}
it holds that
\begin{equation*}
	T_t f(x)=\mathbb{E}[f(X(t)) | X(0)=x]=e^{-\lambda t}f(x).
\end{equation*}
Therefore, 
\begin{equation*}
	\mathbb{E}[\Psi_n(X(t)) | X(0)=x]=e^{-\lambda_n t}\Psi_n(x)
\end{equation*}
since \( \Psi_n(x) \) is an eigenfunction of \( \mathcal{G} \) with corresponding eigenvalue \( \lambda_n \), as specified in \eqref{Bessel_eigenvalue}.
Now, applying Theorem 2.2 (and also Example 3.1) from \cite{patie_srapionyan_2021}, the result follows.
\end{proof}

\begin{remark}
Based on the preceding theorem, if a first-order polynomial is an eigenfunction of a diffusion process, then it is straightforward to show that the process exhibits long-range dependence. However, since this condition does not hold for the Bessel process, we proceed differently.
\end{remark}
To determine the correlation structure of the (fractional) Bessel process, we utilize the spectral representation of its transition density (see  \eqref{Bessel_trans_dens} and Theorem \ref{fracBesselprocess_trans}). To simplify the formulas and reduce their complexity, we introduce the following notation:
\begin{align}
\Psi^{*}_n(x) &= \frac{1}{n+\nu+1/2}\left(\frac{1}{2\beta}\right)^{\nu+1/2}\left(\nu+\frac{3}{2}\right)^{\nu+3/2} \sqrt{\frac{\Gamma(2\nu+n+1)}{n! \Gamma(2\nu+1)}}\Psi_n(x), \\
\Psi^*(x,\lambda)&=\sqrt{\frac{2\nu+1}{2\pi k(\lambda)\Gamma(2\nu+1)}}\frac{|\Gamma\left(\frac{1}{2}+\nu+i\beta/k(\lambda)\right)|}{\left(-2\mu\right)^{\nu+1}}x^{-\nu-1/2}e^{-\mu x+ \pi \beta/(2k(\lambda))}\Psi(x,\lambda),
\end{align}
where $\Psi_n(x)$ and $\Psi(x,\lambda)$ are given by \eqref{eigenfunction_discrete} and \eqref{eigenfunction_essential}.
Since the eigenfunctions $\left(\Psi_n, \n \in \mathbb{N}_0\right)$ form an orthogonal system in $L^{2}\left(\left[ 0, +\infty\right>, \mathfrak{m}\right)$, we have
\begin{equation*}
\int_{\mathbb{R}}\Psi^*_n(y)\Psi^*_m(y)\pi(y)dy	=0, \quad n \neq m.
\end{equation*}
Moreover, $\lambda_0=1$ and $\Psi^*_0(x)=1$. Finally, the transition density takes on a more conventional form (see e.g. \cite{linetsky2004spectral}):
\begin{equation}\label{spec_rep_conv_form}
p(x,t;y)=\pi(x)\left(\sum_{n=0}^{\infty}  \Psi^*_n(x)\Psi^*_n(y)e^{-\lambda_n t} + \int_{\mu^2/2}^{\infty}\Psi^*(x, \lambda)\Psi^*(y, \lambda)e^{-\lambda t} d\lambda\right)
\end{equation}
which clearly demonstrates the exponential rate at which the transition density approaches the stationary density.
\begin{theorem}
Let $\mu<0$, $\nu \in \langle -1, +\infty \rangle \backslash \{-1/2\}$, and let $(X(t), \, t \geq 0)$ be a stationary solution of the SDE \eqref{Bessel_SDE}. Then, for any $t\geq s$ and $y>0$:
\begin{itemize}
\item[i)]\begin{equation}
 \mathbb{E}[X(t) | X(s)= y] =\sum_{n=0}^{\infty} c_n \Psi^*_n(y)e^{-\lambda_n(t-s)} + \int_{\mu^2/2}^{\infty}c(\lambda)\Psi^*(y, \lambda)e^{-\lambda (t-s)} d\lambda,  
\end{equation}
\item[ii)]\begin{equation}\label{corr_Bess_cons_drift}  
 \mathrm{Corr}(X(t),X(s))=\mu^{-2} \left(\frac{\nu}{2}+\frac{1}{2}\right) \biggl[\sum_{n=1}^{\infty} c^2_ne^{-\lambda_n(t-s)} + \int_{\mu^2/2}^{\infty}c(\lambda)^2e^{-\lambda (t-s)} d\lambda\biggr],
\end{equation}
\end{itemize}
where
\begin{equation*}
c_n=\int_{\mathbb{R}} x \Psi^*_n(x)\pi(x)dx, \quad c(\lambda)=\int_{\mathbb{R}}x\Psi^*(x,\lambda)\pi(x) dx.
\end{equation*}
\end{theorem}
\begin{proof}
Taking into account spectral representation \eqref{spec_rep_conv_form}, together with the fact that $(X(t), t \geq 0)$ is a square integrable stochastic process (under the specified parameter values), we have
\begin{align*}
	\mathbb{E}[X(t)\mid X(s)=y] 
	&= \int x\, p(x,t;y)\, dx \\
	&= \int x\, \pi(x) \left( 
	\sum_{n=0}^{\infty} \Psi^*_n(x)\Psi^*_n(y)e^{-\lambda_n t}
	+ \int_{\mu^2/2}^{\infty} \Psi^*(x, \lambda)\Psi^*(y, \lambda)e^{-\lambda t}\, d\lambda 
	\right) dx \\
	&= \sum_{n=0}^{\infty} \left( \int x \Psi^*_n(x)\pi(x)\, dx \right)
	\Psi^*_n(y) e^{-\lambda_n (t-s)} \\
	&\quad + \int_{\mu^2/2}^{\infty} \left( \int x \Psi^*(x, \lambda)\pi(x)\, dx \right)
	\Psi^*(y, \lambda) e^{-\lambda (t-s)}\, d\lambda \\
	&= \sum_{n=0}^{\infty} c_n\, \Psi^*_n(y) e^{-\lambda_n (t-s)} 
	+ \int_{\mu^2/2}^{\infty} c(\lambda)\, \Psi^*(y, \lambda) e^{-\lambda (t-s)}\, d\lambda
\end{align*}
which proves formula $i)$. Next, since the process is stationary, by the law ot total expectation (also known as tower property) we have:
\begin{align*}
\mathbb{E}[X(t)X(s)]&=\mathbb{E}\left[X(s) \mathbb{E}[X(t) | X(s)]\right] \\
           &=\mathbb{E}\left[X(s)\left(\sum_{n=0}^{\infty} c_n\Psi^*_n(X(s))e^{-\lambda_n (t-s)} + \int_{\mu^2/2}^{\infty}c(\lambda)\Psi^*(X(s), \lambda)e^{-\lambda (t-s)} d\lambda  \right)\right] \\
           &=\sum_{n=0}^{\infty} c_n\mathbb{E}[X(0)\Psi^*_n(X(0))]e^{-\lambda_n (t-s)} + \int_{\mu^2/2}^{\infty}c(\lambda)\mathbb{E}[X(0)\Psi^*(X(0), \lambda)]e^{-\lambda (t-s)} d\lambda .
\end{align*}
Taking into account that $\lambda_0=0$, $\Psi^*_0(\cdot)=1$, and 
\begin{equation*}
c_0=\int x \Psi^*_0(x)\pi(x)dx=\int x \pi(x)dx =\mathbb{E}[X(0)],
\end{equation*}
we have
\begin{equation*}
c_0\mathbb{E}[X(0)\Psi^*_0(X(0))]e^{-\lambda_0 (t-s)}=\mathbb{E}[X(0)]^2.
\end{equation*}
Since 
\begin{equation*}
\mathrm{Var}(X(0))=\mu^2 / \left(\nu/2+1/2\right), \quad c_n=\mathbb{E}[X(0)\Psi^*_n(X(0))], \quad c(\lambda)=\mathbb{E}[X(0)\Psi^*(X(0), \lambda)],
\end{equation*} we finally arrive at
\begin{align*}
	\mathrm{Corr}(X(t),X(s))&=\mu^{-2} \left(\frac{\nu}{2}+\frac{1}{2}\right) \biggl[\sum_{n=1}^{\infty} c_n \mathbb{E}[X(0)\Psi^*_n(X(0))]e^{-\lambda_n(t-s)} + \nonumber \\  &+\int_{\mu^2/2}^{\infty}c(\lambda)\mathbb{E}[X(0)\Psi^*(X(0), \lambda)]e^{-\lambda (t-s)} d\lambda\biggr], \\
	&=\mu^{-2} \left(\frac{\nu}{2}+\frac{1}{2}\right) \biggl[\sum_{n=1}^{\infty} c^2_n e^{-\lambda_n(t-s)} +\int_{\mu^2/2}^{\infty}c^2(\lambda)e^{-\lambda (t-s)} d\lambda\biggr],
\end{align*}
which proves $ii)$.
\end{proof}
To the best of our knowledge, the closed-form formula \eqref{corr_Bess_cons_drift} for the correlation function of the Bessel process with constant drift is not available in the existing literature and is presented here as new.
Now we are ready to calculate correlation structure of the fractional counterpart of the process.

\begin{theorem}
	Let $\mu<0$, $\nu \in \langle -1, +\infty \rangle \backslash \{-1/2\}$ and let $(X(t), \, t \geq 0)$ be a stationary solution of SDE \eqref{Bessel_SDE}. Then for any $t \geq s$ and $y>0$:
\begin{align}
				\mathrm{Corr}(X_{\alpha}(t),X_{\alpha}(s))&=\mu^{-2} \left(\frac{\nu}{2}+\frac{1}{2}\right) \biggl[\sum_{n=1}^{\infty} c^2_n \left(\mathcal{E}_{\alpha}(-\lambda_n t^{\alpha})+\frac{\lambda_n \alpha t^\alpha}{\Gamma(1+\alpha)}\int_{0}^{s/t}\frac{\mathcal{E}_{\alpha}(-\lambda_n t^{\alpha}(1-z)^{\alpha})}{z^{1-\alpha}}dz\right) + \nonumber \\ \label{corr_frac_Bess_cons_drift} &+\int_{\mu^2/2}^{\infty}c^2(\lambda)\left(\mathcal{E}_{\alpha}(-\lambda t^{\alpha})+\frac{\lambda \alpha t^\alpha}{\Gamma(1+\alpha)}\int_{0}^{s/t}\frac{\mathcal{E}_{\alpha}(-\lambda t^{\alpha}(1-z)^{\alpha})}{z^{1-\alpha}}dz\right) d\lambda\biggr],
\end{align}
	where
	\begin{equation}
		c_n=\int_{\mathbb{R}} x \Psi^*_n(x)\pi(x)dx, \quad c(\lambda)=\int_{\mathbb{R}}x\Psi^*(x,\lambda)\pi(x) dx.
	\end{equation}
\end{theorem}
\begin{proof}
The proof is based on the technique first used in [Theorem 3.1, \cite{LeonenkoMeerschaertSikorskii_CSFPD_2013}]. In particular, we rely on the Lebesgue-Stieltjes integration formula
\begin{equation}\label{Leb-Stie-Biv}
\mathrm{Corr}(X_{\alpha}(t),X_{\alpha}(s))=\int_{0}^{\infty}\int_{0}^{\infty}\mathrm{Corr}(X(u), X(v))H(du,dv)
\end{equation} with respect to bivariate distribution function 
\begin{equation*}
H(u,v)=P(E(t) \leq u , E(s) \leq v).
\end{equation*}
Observe that the time components $u$ and $v$ in $\mathrm{Corr}(X(u),X(v))$ appear only through an exponential expression of the form  $e^{-\theta|u-v|}$.
According to [Theorem 3.1, \cite{LeonenkoMeerschaertSikorskii_CSFPD_2013}]
\begin{equation}\label{Leb-Stie-Biv2}
\int_{0}^{\infty}\int_{0}^{\infty}e^{-\theta |u-v|}H(du,dv)=\left(\mathcal{E}_{\alpha}(-\theta t^{\alpha})+\frac{\theta \alpha t^\alpha}{\Gamma(1+\alpha)}\int_{0}^{s/t}\frac{\mathcal{E}_{\alpha}(-\theta t^{\alpha}(1-z)^{\alpha})}{z^{1-\alpha}}dz\right).
\end{equation}
Taking into account \eqref{corr_Bess_cons_drift} and  \eqref{Leb-Stie-Biv2}, a straightforward application of Fubini-Tonelli theorem in \eqref{Leb-Stie-Biv} yields the desired formula \eqref{corr_frac_Bess_cons_drift}.
\end{proof}
\begin{remark}
If $t=s$ both formulas \eqref{corr_Bess_cons_drift} and \eqref{corr_frac_Bess_cons_drift} must yield the value $1$. To see this, simply recall Parseval's theorem in the context of spectral theory for self-adjoint operator $\mathcal{G}$ in $L^{2}\left(\left[ 0, +\infty\right>, \mathfrak{m}\right)$, which ensures
\begin{equation}\label{Parseval_iden}
\mathbb{E}[X(0)^2]= \sum_{n=0}^{\infty} c^2_n + \int_{\mu^2/2}^{\infty}c(\lambda)^2 d\lambda
\end{equation}
(for details, see e.g. \cite{linetsky2004spectral}). Therefore
\begin{align*}
\mathrm{Corr}(X(t), X(t))&=\mu^{-2} \left(\frac{\nu}{2}+\frac{1}{2}\right) \biggl[\sum_{n=1}^{\infty} c^2_ne^{-\lambda_n(t-t)} + \int_{\mu^2/2}^{\infty}c(\lambda)^2e^{-\lambda (t-t)} d\lambda\biggr] \\
&=\mu^{-2} \left(\frac{\nu}{2}+\frac{1}{2}\right) \biggl[\sum_{n=0}^{\infty} c^2_n + \int_{\mu^2/2}^{\infty}c(\lambda)^2 d\lambda-c_0^2\biggr] \\
&=\mu^{-2} \left(\frac{\nu}{2}+\frac{1}{2}\right) \biggl [\mathbb{E}[X(0)^2]-\mathbb{E}[X(0)]^2\biggr] \\
&=\mu^{-2} \left(\frac{\nu}{2}+\frac{1}{2}\right) \mathrm{Var}(X(0)) \\
&=1.
\end{align*}
In the fractional case we have (see [Remark 3.2., \cite{LeonenkoMeerschaertSikorskii_CSFPD_2013}])
\begin{equation*}
\frac{\theta \alpha t^\alpha}{\Gamma(1+\alpha)}\int_{0}^{1}\frac{\mathcal{E}_{\alpha}(-\theta t^{\alpha}(1-z)^{\alpha})}{z^{1-\alpha}}dz=1-\mathcal{E}_{\alpha}(-\theta t^{\alpha}),
\end{equation*}
so that
\begin{align*}
	\mathrm{Corr}(X_{\alpha}(t),X_{\alpha}(s))&=\mu^{-2} \left(\frac{\nu}{2}+\frac{1}{2}\right) \biggl[\sum_{n=1}^{\infty} c^2_n \left(\mathcal{E}_{\alpha}(-\lambda_n t^{\alpha})+1-\mathcal{E}_{\alpha}(-\lambda_n t^{\alpha})\right) + \nonumber \\  &+\int_{\mu^2/2}^{\infty}c^2(\lambda)\left(\mathcal{E}_{\alpha}(-\lambda t^{\alpha})+1-\mathcal{E}_{\alpha}(-\lambda t^{\alpha})\right) d\lambda\biggr] \\
	&=\mu^{-2} \left(\frac{\nu}{2}+\frac{1}{2}\right) \biggl[\sum_{n=1}^{\infty} c^2_n  +\int_{\mu^2/2}^{\infty}c^2(\lambda) d\lambda\biggr] \\
	&=1.
\end{align*}
\end{remark}
Now that we have obtained explicit formulas for the correlation structure of fractional Bessel process, we proceed by showing the correlation function falls of like a power law of order $0 <\alpha <1$, which corresponds to the order of the Caputo fractional derivative. In other words, the fractional Bessel process exhibits long-range dependence property.
\begin{theorem}
Let $\mu<0$, $\nu \in \langle -1, +\infty \rangle \backslash \{-1/2\}$ and let $(X(t), \, t \geq 0)$ be a stationary solution of the SDE \eqref{Bessel_SDE}. Then, for any fixed $s>0$ as $t \to \infty$:
\begin{itemize}
	\item[i)]\begin{equation} 
		\mathrm{Corr}(X(t),X(s)) \sim \mu^{-2} \left(\frac{\nu}{2}+\frac{1}{2}\right) c_1 e^{-\lambda_1(t-s)},
	\end{equation}
	\item[ii)]\begin{equation} 
		\mathrm{Corr}(X_{\alpha}(t),X_{\alpha}(s)) \sim  \frac{1}{t^{\alpha}\Gamma(1-\alpha)}\left(\mu^{-2} \left(\frac{\nu}{2}+\frac{1}{2}\right)\left(\sum_{n=1}^{\infty} \frac{c_n^2}{\lambda_n}+\int_{\mu^2/2}^{\infty}\frac{c^2(\lambda)}{\lambda}\right)+ \frac{s^{\alpha}}{\Gamma(1+\alpha)}\right).
	\end{equation}
\end{itemize}
\end{theorem}
\begin{proof}
Recall that $\left(\lambda_n, \, n \in \mathbb{N}_0\right)$ is an increasing sequence of eigenvalues and due to assumption we have $E[X(0)^2] < \infty$ so that summands in \eqref{Parseval_iden} are finite. Therefore, from \eqref{corr_Bess_cons_drift} we have

\begin{align*}
\lim_{t \to \infty}\frac{\mathrm{Corr}(X(t),X(s))}{\mu^{-2} \left(\frac{\nu}{2}+\frac{1}{2}\right) c_1 e^{-\lambda_1(t-s)}}&=\lim_{t \to \infty}\left[1+\sum_{n=2}^{\infty}c_n^2 e^{-\left(\lambda_n-\lambda_1\right)\left(t-s\right)}+\int_{\mu^2/2}^{\infty}c^2(\lambda)e^{-\left(\lambda-\lambda_1\right)\left(t-s\right)} d \lambda\right] \\
 &=1,
\end{align*}
where the second equality is justified by dominated convergence argument. This proves $i)$. 

Next, we prove the second asymptotic formula. Taking into account Theorem 3.1. and Remarks 3.2. and 3.3. from \cite{LeonenkoMeerschaertSikorskii_CSFPD_2013}, for $\theta>0$ and fixed $s>0$ we have:
\begin{equation*}
\mathcal{E}_{\alpha}(-\theta t^{\alpha})+\frac{\theta  \alpha t^\alpha}{\Gamma(1+\alpha)}\int_{0}^{s/t}\frac{\mathcal{E}_{\alpha}(-\theta  t^{\alpha}(1-z)^{\alpha})}{z^{1-\alpha}}dz \sim \frac{1}{t^{\alpha}\Gamma(1-\alpha)}\left(\frac{1}{\theta}+\frac{s^{\alpha}}{\Gamma(\alpha+1)}\right), \quad t \to \infty
\end{equation*}
and for any $t\geq s>0$
\begin{equation*} \mathcal{E}_{\alpha}(-\theta t^{\alpha})+\frac{\theta  \alpha t^\alpha}{\Gamma(1+\alpha)}\int_{0}^{s/t}\frac{\mathcal{E}_{\alpha}(-\theta  t^{\alpha}(1-z)^{\alpha})}{z^{1-\alpha}}dz  \leq 1.
\end{equation*}
Now, \eqref{corr_frac_Bess_cons_drift} together with dominated convergence theorem yields
\begin{align*}
	\mathrm{Corr}(X_{\alpha}(t),X_{\alpha}(s))&=\mu^{-2} \left(\frac{\nu}{2}+\frac{1}{2}\right) \biggl[\sum_{n=1}^{\infty} c^2_n \left(\mathcal{E}_{\alpha}(-\lambda_n t^{\alpha})+\frac{\lambda_n \alpha t^\alpha}{\Gamma(1+\alpha)}\int_{0}^{s/t}\frac{\mathcal{E}_{\alpha}(-\lambda_n t^{\alpha}(1-z)^{\alpha})}{z^{1-\alpha}}dz\right)+ \\  &+\int_{\mu^2/2}^{\infty}c^2(\lambda)\left(\mathcal{E}_{\alpha}(-\lambda t^{\alpha})+\frac{\lambda \alpha t^\alpha}{\Gamma(1+\alpha)}\int_{0}^{s/t}\frac{\mathcal{E}_{\alpha}(-\lambda t^{\alpha}(1-z)^{\alpha})}{z^{1-\alpha}}dz\right) d\lambda\biggr] \\
	&\sim \mu^{-2} \left(\frac{\nu}{2}+\frac{1}{2}\right) \biggl[\sum_{n=1}^{\infty}  \frac{1}{t^{\alpha}\Gamma(1-\alpha)}\left(\frac{1}{\lambda_n}+\frac{s^{\alpha}}{\Gamma(\alpha+1)}\right) c^2_n+ \\  &+\int_{\mu^2/2}^{\infty}\frac{1}{t^{\alpha}\Gamma(1-\alpha)}\left(\frac{1}{\lambda}+\frac{s^{\alpha}}{\Gamma(\alpha+1)}\right) c^2(\lambda) d\lambda\biggr], \quad t \to \infty. 
\end{align*}
Finally, taking into account \eqref{Parseval_iden}, we obtain
\begin{align*}
	\mathrm{Corr}(X_{\alpha}(t), X_{\alpha}(s)) 
	&\sim \mu^{-2} \left( \frac{\nu}{2} + \frac{1}{2} \right)
	\frac{1}{t^{\alpha} \Gamma(1 - \alpha)} \biggl(
	\sum_{n=1}^{\infty} \frac{c_n^2}{\lambda_n}
	+ \int_{\mu^2/2}^{\infty} \frac{c^2(\lambda)}{\lambda}\, d\lambda \\
	&\hspace{5em}
	+ \mathrm{Var}(X(0)) \frac{s^{\alpha}}{\Gamma(\alpha + 1)}
	\biggr), \quad t \to \infty,
\end{align*}
which proves $ii)$.
\end{proof}

The obtained analytical properties equip the fractional Bessel process with powerful modeling capabilities. In particular, its ability to capture long-range dependence and subdiffusive dynamics makes it a proper candidate for describing phenomena in applied fields such as queueing theory, finance, biology and physics. The flexibility offered by the fractional time-change opens the door to modeling systems influenced by memory effects or random delays, thereby extending the practical relevance of the classical Bessel process to a broader class of real-world scenarios. 

\section{Application to Queueing theory}\label{Applications}
In this section, we explain how the fractional Bessel process can be applied in queueing theory.

Time-fractional models of the form \eqref{fracBesselprocess} are preferred when there are random periods during which the process becomes trapped (i.e., temporarily halted) due to possibly unknown environmental forces. In queueing theory, this corresponds to scenarios where the server may malfunction and stop working for random intervals of time, such that the total unfinished work in the system remains unchanged until the server is repaired. Here, we assume that no new customers arrive during the repair period.

The fractional Bessel process introduced in the previous section is naturally connected to queueing theory. In particular, the (non-fractional) Bessel process arises as the heavy-traffic limit of the unfinished work in a so-called polling system. This system consists of two queues, each visited by a single server one at a time, with the server working until no work is left in the current queue. Switchover times from one queue to another are assumed to be nonzero (see \cite{Coffman1998} for details). Moreover, Coffman et al.~proved the Heavy-Traffic Averaging Principle (HTAP) for this model.

\begin{remark}
	It should be noted that the assumption of nonzero switchover times is essential for obtaining the Bessel process in the heavy-traffic regime. If this assumption is violated, the limiting process becomes a reflected Brownian motion instead (see \cite{coffman1995polling}). Therefore, the Bessel process serves as a cornerstone in modeling heavy-traffic regimes in queueing theory under nonzero switchover times.
\end{remark}

Next, we show how these results can be extended to a more general model, namely, to the fractional case. To that end, we first summarize the Coffman two-queue polling model with nonzero switchover times and its key assumptions and results.

We consider $n$ two-queue polling systems. For each system, we use the following notation for the $l$th queue, where $l \in \{1, 2\} $ (as in \cite{Coffman1998}):

\begin{itemize}
	\item \( \xi^{n,l}_k > 0 \): interarrival time for the \( k \)th customer to the \( l \)th queue,
	\item \( \tau^{n,l}_i = \sum_{k=1}^{i} \xi^{n,l}_k \): time of the \( i \)th arrival to the \( l \)th queue,
	\item \( \eta^{n,l}_i \): \( i \)th service time in the \( l \)th queue,
	\item \( s^{n,l}_i \): \( i \)th switchover time from the \( l \)th queue.
\end{itemize}

We assume that the sequences \( \{ \xi^{n,l}_k,\, k \in \mathbb{N} \} \), \( \{ \eta^{n,l}_k,\, k \in \mathbb{N} \} \), and \( \{ s^{n,l}_k,\, k \in \mathbb{N} \} \) are independent and identically distributed (i.i.d.) for each \( l \in \{1,2\} \). For the expected values of these random variables, we have:

\begin{equation*}
 \lambda^n_l=\frac{1}{\mathbb{E}\xi^{n,l}_1}, \quad \mu^n_l=\frac{1}{\mathbb{E}\eta^{n,l}_1}, \quad d^{n}_l=\mathbb{E}s^{n,l}_1, \quad \rho^n_l=\frac{\lambda^{n}_{l}}{\mu^{n}_l}, \quad \rho^n=\rho^n_1+\rho^n_2, \quad \sigma^n_l=\sqrt{\mathbb{E}\left(\eta^{n,l}_1-\rho^n_l \cdot \xi^{n,l}_1 \right)^2}.
\end{equation*}
Let
\begin{equation}\label{TotalUnfinished_n}
	V^{n}(t) = \frac{1}{\sqrt{n}} U^n_{nt}, \quad t \geq 0,
\end{equation}
where \( U^n_t \) denotes the total unfinished work in the \( n \)th system at time \( t \). It is assumed that the initial unfinished work \( U^n_0 \) is independent of the sequences \( \{ \xi^{n,l}_k,\, k \in \mathbb{N} \} \), \( \{ \eta^{n,l}_k,\, k \in \mathbb{N} \} \), and \( \{ s^{n,l}_k,\, k \in \mathbb{N} \} \) for \( l \in \{1,2\} \).

[Theorem 1, \cite{Coffman1998}] establishes that the scaled total unfinished work \eqref{TotalUnfinished_n} in the two-queue polling system can be approximated in the heavy-traffic regime by a Bessel process with constant drift. It turns out that this heavy-traffic approximation in the non-fractional case can be naturally extended to the fractional case.

Let
\begin{equation}\label{HTApproxProcesses}
	V^{n}_{\alpha}(t) = V^n(E(t)), \quad V_{\alpha}(t) = V(E(t)), \quad \alpha \in \langle 0, 1 \rangle,
\end{equation}
where \( V = \{V(t),\, t \geq 0\} \) is the solution of the stochastic differential equation
\[
dV(t) = \left( c + \frac{d}{V(t)} \right) dt + \sigma\, dB(t),
\]
and \( E(t) \) is the inverse of an \( \alpha \)-stable subordinator as defined in \eqref{InverseStabDef}, assumed to be independent of the processes \( V^n \) and \( V \).

\begin{theorem}\label{CoffmanHTApprox}
Assume that $V^n(0) \Rightarrow V_0$ for some nonnegative random variable $V_0$. Suppose \begin{equation*}
	\lambda^n_l \to \lambda_l, \quad \mu^n_l \to \mu_l, \quad \sigma^n_l \to \sigma_l, \quad d^n_l \to d_l, \quad n \to \infty,
\end{equation*}
and define \begin{equation*}
	\sigma^2=\lambda_1 \sigma^2_1 + \lambda_2 \sigma^2_2 >0, \quad \rho_l=\lambda_l/\mu_l, \quad l \in \{1,2\}, \quad  d=\rho_1 \rho_2\left(d_1+d_2\right).
\end{equation*}
 Under the heavy traffic condition
\begin{equation*}
	\sqrt{n}(\rho^n-1) \to c, \quad n \to \infty
\end{equation*}
and the Lindeberg conditions
\begin{equation*}
	\lim_{n \to \infty}\mathbb{E}\left(\xi^{n,l}_1\right)^2 \cdot 1_{\{\xi^{n,l}_1>\epsilon \sqrt{n}\}}=\lim_{n \to \infty}\mathbb{E}\left(\eta^{n,l}_1\right)^2 \cdot 1_{\{\eta^{n,l}_1>\epsilon \sqrt{n}\}}=\lim_{n \to \infty}\mathbb{E}\left(s^{n,l}_1\right)^2 \cdot 1_{\{s^{n,l}_1>\epsilon \sqrt{n}\}}=0,
\end{equation*}
we have \begin{equation*}
	V^n_{\alpha} \Rightarrow V_{\alpha}, \quad n \to \infty
\end{equation*}
in the Skorokhod space $\mathbb{D}(\left[0,\,+\infty\right>)$ equipped with the $J_1$ topology.
 
\end{theorem}
\begin{proof}
Under the conditions of the theorem, we have
 \begin{equation*}
	V^n \Rightarrow V, \quad n \to \infty
\end{equation*}
in the Skorokhod space $\mathbb{D}(\left[0,\,+\infty\right>)$ with the $J_1$ topology (see [Theorem 1., \cite{Coffman1998}]). Since the inverse $\alpha$-stable subordinator $(E(t), \, t \geq 0)$ is a.s.~continuous, non-decreasing,  and independent of both $V^n$ and $V$, by classical arguments (\cite{whitt2002stochastic}) we have the desired convergence
\begin{equation*}
	V^n_{\alpha} \Rightarrow V_{\alpha}, \quad n \to \infty
\end{equation*}
in the Skorokhod space $\mathbb{D}(\left[0,\,+\infty\right>)$ with $J_1$ topology.
\end{proof}

The idea behind the random time-change in the processes \eqref{HTApproxProcesses} via the inverse of the $\alpha$-stable subordinator is as follows. If the server in the aforementioned polling system malfunctions and stops working for random periods of time, then the total unfinished work remains the same until the server is repaired and new customers are allowed to enter the system. 

The parameter $\alpha$ controls the length of the malfunction periods in the following way:
\begin{itemize}
	\item Smaller values of $\alpha$ correspond to longer malfunction periods, i.e.~extended intervals during which the total unfinished work remains constant.
	\item Larger values of $\alpha$  correspond to shorter malfunction periods, i.e.~brief intervals during which the total unfinished work remains constant.
	\item  Setting $\alpha=1$ recovers the classical result of [Theorem 1., \cite{Coffman1998}].
\end{itemize}
This means that the heavy traffic regime with random malfunction times can be modeled by stochastic process $V^{\alpha}(t)=\{\sigma X_{\alpha}(t), \, t \geq 0\}$, where $X_{\alpha}(t)$ is defined by \eqref{fracBesselprocess} and the corresponding parameters from \eqref{Bessel_SDE} are given by
\begin{equation*}
	\mu=\frac{c}{\sigma}, \quad \nu=\frac{d}{\sigma^2}-\frac{1}{2}.
\end{equation*} 
Therefore, one can exploit the spectral representation result in Theorem \ref{fracBesselprocess_trans_theorem} to compute the asymptotic distributional properties of the two-queue polling model with exhaustive service and possible random malfunction times.
Under the assumptions of Theorem \ref{CoffmanHTApprox}, and for any real-valued continuous function $f$, the well-known heavy traffic averaging principle (HTAP) is valid (see [Theorem 2.2., \cite{Coffman1998}]):
\begin{equation*}
	\int_{0}^{t}f(V^{n,l}_s)ds \Rightarrow \int_{0}^{t}\int_{0}^{1}f(u V_s)du ds, \quad n \to \infty.
\end{equation*}
It turns out that the same reasoning used in the previous theorem extends this result to the time-changed model as well:

\begin{equation*}
	\int_{0}^{E(t)}f(V^{n,l}_s)ds \Rightarrow \int_{0}^{E(t)}\int_{0}^{1}f(u V_s)du ds, \quad n \to \infty.
\end{equation*}

\bigskip \bigskip

\textbf{Acknowledgements} \newline
This work was supported by the Croatian Science Foundation under the project number HRZZ-IP-2022-10-8081. \\

\newpage

\bibliographystyle{agsm}

\bibliography{FBP_references}

\end{document}